\documentclass[a4paper, 11pt]{article}
\usepackage[margin=1in]{geometry}
\usepackage{amsmath, amssymb,amsthm,url,mathrsfs,graphicx,amscd,amsfonts}
\usepackage[mathscr]{eucal}
\usepackage{dsfont}
\usepackage{mathtools}
\usepackage[all]{xy}

\usepackage[utf8]{inputenc}
\usepackage{graphicx,todonotes,hyperref}
\usepackage{epstopdf}
\usepackage{inputenc}
\usepackage{enumitem}
\usepackage{amsmath}
\usepackage{graphicx}
\usepackage{xcolor}
\usepackage{xspace}
\usepackage{amssymb}
\usepackage{amsfonts}
\usepackage{amssymb}
\usepackage{graphicx}
\usepackage{mathtools}
\usepackage{amsthm}
\usepackage{tikz-cd}
\usepackage{multirow}
\usepackage{enumitem}

\input xy
\xyoption{all}

\newtheorem*{acknowledgements*}{Acknowledgements}

\allowdisplaybreaks
\theoremstyle{definition}
\newtheorem{theorem}{Theorem}[section]
\newtheorem{prop}[theorem]{Proposition}
\newtheorem{lemma}[theorem]{Lemma}
\newtheorem{cor}[theorem]{Corollary}

\newtheorem{defn}[theorem]{Definition}

\newtheorem*{theorem*}{Theorem}
\newtheorem*{corollary*}{Corollary}
\newtheorem*{prop*}{Proposition}
\newtheorem*{rmk*}{Remark}

\begin{document}

\title{Geodesic divergence on Riemannian planes with bounded geometry}
\author{Ritvik Saharan}
\date{}
\maketitle

\begin{abstract}
In this article, we study Riemannian planes $(M,g)$ which satisfies a bounded geometry condition ( bounded sectional curvature and a non-vanishing injectivity radius of $(M,g)$).
We show that such a Riemannian plane is always quasi-isometric to a radially symmetric plane which also satisfies the aforesaid bounded geometry condition.
Using that, we study the quasi-redirecting boundary of Riemannian planes with bounded geometry and show that it is either a single point or $\mathbb{S}^1$.
We contrast the quasi-redirecting boundary with the Martin boundary and construct Riemannian planes with bounded geometry for which the 
We also characterise quasi-redirection in Riemannian planes with bounded geometry via their asymptotic cones and their cut points.
Lastly, we characterise $\mathbb{E}^2$ and $\mathbb{H}^2$ upto quasi-isometry by their asymptotic cones.
\end{abstract}

\maketitle
\section{Introduction}
\label{sec: Intro}
Unlike in dimensions $n \geq 3$ wherein for any one-ended locally finite graph $G$, we can find a Riemannian manifold $(M,g)$ diffeomorphic to $\mathbb{R}^n$ and quasi-isometric to $G$  ( see proposition 5.3 of \cite{asymp}), the Riemannian geometry of simply connected manifolds of dimension 2 is quite constrained as such an $(M,g)$ is by its very nature quasi-isometric to planar graphs and one cannot make them quasi-isometric to a metric space which is not quasi-isometric to a planar graph.

We shall study Riemannian planes which has bounded geometry.
To be more precise, their sectional curvatures are bounded and whose injectivity radius is bounded below by a positive number.
From theorems 1.1 and 1.2 of \cite{Bowditch2}, we see that they admit a uniformly bilipschitz smooth triangulation.

We use that to show that every Riemannian plane $(M,g)$ with bounded geometry is quasi-isometric to a radially symmetric Riemannian plane $(N,h)$ with bounded geometry.

\begin{theorem}
If $(M,g)$ is a Riemannian plane with bounded geometry, then $(M,g)$ is quasi-isometric and conformally equivalent to a radially symmetric Riemannian plane $(N,h)$ which also has bounded geometry .
\end{theorem}

Therefore, we do not lose any generality by working with radially symmetric planes as every Riemannian plane with bounded geometry is quasi-isometric  to a radially symmetric plane with bounded geometry.

We recall quasi-redirection which was introduced by Qing and Rafi in \cite{QingRafi} to generalise the notions of Gromov hyperbolicity and Gromov boundaries to a much wider class of metric spaces which are bordified with the quasi-redirecting boundary.
Corollary 6.9 of \cite{QingRafi} tells that if the metric space in question is Gromov hyperbolic, then the quasi-redirecting boundary and the quasi-redirecting compactification is indeed the Gromov boundary.

We study quasi-redirection in Riemannian planes with bounded geometry and show the following theorem.
\begin{theorem}
Let $(M,g)$ be a Riemannian plane with bounded geometry. Then 
\begin{itemize}
    \item All quasi-geodesic rays in $(M,g)$ redirect with each other and the quasi-redirecting boundary $\partial M$ is a single point.
    \item $(M,g)$ satisfies all the quasi-redirecting assumptions 0-3 ( as stated in \cite{QingRafi}) and the quasi-redirecting boundary $\partial M$ is homeomorphic to $\mathbb{S}^1$.
\end{itemize}
\end{theorem}

We also study the Martin boundaries $\partial_\Delta M$ of Riemannian planes with bounded geometry $(M,g)$ and note that $\partial_{\Delta} M$ is a point if $(M,g)$ is conformally equivalent to $\mathbb{E}^2$ and $\partial_{\Delta} M$ is homeomorphic to $\mathbb{S}^1$ if $(M,g)$ is conformally equivalent to $\mathbb{H}^2$.
We observe that the quasi-redirecting boundary need not coincide with the Martin boundary.
\begin{theorem}
    \begin{itemize}
        \item There is a Riemannian plane with bounded geometry $(M,g)$ whose $\partial M$ is a single point but $\partial_{\Delta} M$ is homeomorphic to $\mathbb{S}^1$.
        \item  There is a Riemannian plane with bounded geometry $(M,g)$ whose $\partial_\Delta M$ is a single point but $\partial M$ is homeomorphic to $\mathbb{S}^1$.
    \end{itemize}
\end{theorem}

We shall also study asymptotic cones of Riemannian planes with bounded geometry $(M,g)$
Asymptotic cones being preserved upto homeomorphisms under quasi-isometries allows us to assume that $(M,g)$ is radially symmetric too with $p \in (M,g)$ being its pole.
By defining $\mathrm{p}= (p_n) \in M^{\mathbb{N}}$ where $p_n = p$ for all $n\in \mathbb{N}$, we observe that- 
\begin{theorem}
\begin{itemize}
    \item If $\partial M$ is a single point, then there exists a non-principal ultrafilter $\omega$ and a sequence of scaling factors $(r_n)$ with $\lim_\omega r_n = \infty$ such that $\mathrm{p} \in \mathrm{Con}_\omega (M, \mathrm{p}, (r_n))$ is not a cut point.
    \item If $\partial M$ is $\mathbb{S}^1$, then for every non-principal ultrafilter $\omega$ and $(r_n)$ being a sequence of scaling factors with $\lim_\omega r_n = \infty$, then $\mathrm{p} \in \mathrm{Con}_\omega (M, \mathrm{p}, (r_n))$. 
\end{itemize}
\end{theorem}

Lastly, we characterise $\mathbb{E}^2$ and $\mathbb{H}^2$ upto quasi-isometry.

\begin{theorem}
Let $(M,g)$ be a Riemannian plane with bounded geometry.
\begin{itemize}
    \item $(M,g)$ is quasi-isometric to $\mathbb{E}^2$ if and only if every asymptotic cone of $(M,g)$ is a proper geodesic metric space without cut points.
    \item $(M,g)$ is quasi-isometric to $\mathbb{H}^2$ if it has an asymptotic cone which is an $\mathbb{R}-$tree which is not proper.
\end{itemize}
\end{theorem}

\section{Background Material}
\subsection{Riemannian Geometry}
From now onwards, unless stated otherwise; every Riemannian manifold $(M,g)$ will be assumed to be complete, non-compact, simply connected and two-dimensional ( or in other words $(M,g)$ is complete and that it is diffeomorphic to $\mathbb{R}^2$) . 

\begin{defn}
Two Riemannian manifolds $(M,g)$ and $(N,h)$ are said to be conformally equivalent if there exists a diffeomorphism $f : M \to N$ with $h_{f(p)} = \exp(2u) g_p$ for some smooth $u : M \to \mathbb{R}$.
\end{defn}

\begin{defn}
Suppose that $M$ is a smooth manifold.
Let a derivation $D : C^\infty(M) \to \mathbb{R}$ which satisfies
\begin{equation*}
    D(fg) =D(f) \cdot g + f\cdot D(g) \forall f, g \in C^\infty(M) 
\end{equation*}
The space of all derivations $T_pM$ at a $p \in M$ has a vector space structure given by
\begin{equation*}
    (D_1+D_2) (f) = D_1(f) + D_2(f)
\end{equation*}
and 
\begin{equation*}
    (\lambda \cdot D) (f) = \lambda \cdot D(f)
\end{equation*}
We call $T_p M$ the \textit{tangent space} at $p \in M$.
\end{defn}

\begin{defn}
Let $(M,g)$ be a Riemannian manifold and choose $p \in (M,g)$.
The \textit{exponential map} $\exp_p : T_p M \to M$ is the map where $\exp_p (0) = p$.
If $v \neq 0$ then $ \exp_p(v) = \gamma(||v||)$ where $\gamma$ is the unique geodesic with $\gamma(0) = p$ and $\gamma'(0) = \frac{v}{||v||}$.
\end{defn}

\begin{defn}
Let $(M,g)$ be a Riemannian plane and let $p \in M$ be such that the exponential map $\exp_p : T_pM \to M$ is a diffeomorphism.

This allows us to parametrise $(M,g)$ and $T_pM$ with radial co-ordinates.
We say that $(M,g)$ is radially symmetric with respect to $p$ if for all $\theta, \varphi \in \mathbb{S}^1$, there is $h \in \mathrm{Isom}(M)$ ( the isometry group of $(M,g)$) such that $h ((r, \theta)) = (r, \varphi)$ for all $r > 0$.
\end{defn}

\textbf{Remark}-Whenever $(M,g)$ is said to be radially symmetric, it is to be understood that $(M,g)$ is radially symmetric with respect to some $p \in (M,g)$.
This $p \in (M,g)$ is called the pole of $(M,g)$ as $\exp_p: T_pM \to M$ is a diffeomorphism.

In such a manifold $(M,g)$ with the metric $g = dr^2 + f(r)^2d\theta^2$, f(r) is smooth with $f(0) =0$ and $f'(0) = 1$ .
Such an $(M,g)$ can be parametrised with polar co-ordinates with $(0,0)$ corresponding to $p$.

\begin{defn}
If $(M,g)$ is a Riemannian manifold, then the \textit{injectivity radius} at $p \in M$, $\mathrm{inj}(p)$ is the greatest $r > 0$ such that $\exp_p (B(0,r))$ is diffeomorphic to its image in $(M,g)$.

We shall denote the injectivity radius at a point $p \in (M,g)$ as $\mathrm{inj}(p)$.
We can also define the injectivity radius for $(M,g)$ $\mathrm{inj}(M)$ where 
\begin{equation*}
    \mathrm{inj}(M) = \inf_{p \in M} \mathrm{inj}(p)
\end{equation*}
\end{defn}

\begin{defn}(IX.6 \cite{Chavel})
Let $(M,g)$ be a complete Riemannian manifold. A set $A \subset M$ is said to be 
\begin{itemize}
    \item \textit{convex} if for any $p, q \in A$, there is a geodesic segment $\gamma_{pq} \subseteq A$ joining $p$ and $q$ such that $\gamma_{pq}$ is the unique minimiser in $M$

    \item \textit{strongly convex} if for any $p, q \in A$, there is a geodesic segment $\gamma_{pq} \subseteq A$ joining $p$ and $q$ such that $\gamma_{pq}$ is the unique minimiser in $M$ and $\gamma_{pq}$ is the only geodesic segment in $A$ joininG $p$ and $q$

    \item \textit{convexity radius }- For a point $p \in M$, define the convexity radius $\mathrm{conv}(p)$
    \begin{equation*}
        \mathrm{conv}(p) = \sup \{ \rho : B(x,r) \text{ is convex for all } r \leq \rho \}
    \end{equation*}

    \item We may also define the convexity radius for $(M,g)$ as a whole as follows
    \begin{equation*}
        \mathrm{conv}(M) = \inf_{p \in M} \mathrm{conv}(M)
    \end{equation*}
\end{itemize}
\end{defn}

\begin{defn}
Let $\gamma : I \to M$ be a geodesic where $I \subset \mathbb{R}$ is connected.
A vector field $J$ along $\gamma$ is said to be a Jacobi field if
\begin{equation*}
    J'' + R(J, \gamma')\gamma'= 0
\end{equation*}
where $R$ is the Riemann curvature tensor.
\end{defn}

\subsection{Metric Spaces}
\begin{defn}
Let $(X,d_X)$ and $(Y,d_Y)$ be geodesic metric spaces.
We say that a map $f:(X,d_X) \to (Y,d_Y)$ is a \textbf{quasi-isometry} if there exists $(q,Q) \in [1, \infty) \times [0, \infty)$ such that $\frac{1}{q}d_X(x_1,x_2) - Q \leq d_Y(f(x_1),f(x_2)) \leq q d_X(x_1, x_2) + Q $ for any $x_1, x_2 \in X$.
\end{defn}

\begin{defn} (\cite{BuragoIvanov}, definition 7.4.13)
We say that a class $\mathfrak{X}$ of compact metric spaces is uniformly totally bounded if 
\begin{itemize}
    \item There is a constant $D$ such that $\mathrm{diam} (X) \leq D$ for all $X \in \mathfrak{X}$.
    \item For every $\epsilon > 0$, there exists a natural number $N= N(\epsilon)$ such that every $X\in \mathfrak{X}$ contains an $\epsilon-$net consisting of no more than $N$ points
\end{itemize}
\end{defn}

\begin{defn}
Let $(X,d)$ be a metric space.
A subset $\mathfrak{X} \subset X$ is said to be an $\epsilon-$net for some $\epsilon > 0$ if $d(x, \mathfrak{X}) \leq \epsilon$ for all $x \in X$.
\end{defn}

\textbf{Busemann functions}- Let $(X,d)$ be a proper, geodesic metric space.
Fix a base point $o \in X$. For each $y$ in $X$, we define a Lipschitz map $b_y : X \to \mathbb{R}$ as follows. 
\begin{equation*}
  b_y(x) = d(x,y) -d(o,y)  
\end{equation*}
This family of $1-$Lipschitz functions induces an embedding $\iota$ of $X$ into $C(X,o)$ ( the space of continuous functions on $X$ which vanish on $o$ endowed with the compact-open topology).
\begin{equation*}
\iota(y) = b_y \forall y \in (X,d)
\end{equation*}
From the Arzela-Ascoli theorem, it follows that the closure of $\iota(X)$ in $C(X,o)$ yields a compact metrisable space $\overline{X}_\infty$ which we shall call the horofunction compactification of $(X,d)$ and we shall denote $\overline{X}_\infty \setminus X$ as $\partial_\infty X$.

\begin{defn}
     We say that $b \in \partial_\infty M$ is a \textit{Busemann point} if there exists a geodesic ray $\gamma : [0, \infty) \to X$ such that $\lim_{t \to \infty} b_{\gamma(t)} = b$.
\end{defn}

\begin{defn}
 Let $(X,d)$ be a metric space.
We say that $(X,d)$ is $\delta-$hyperbolic for $\delta \geq 0$ if for any four points $a,b,c,p \in X$, the following inequality holds
\begin{equation*}
    (b,c)_p \geq \min ((a,b)_p, (a,c)_p)-\delta
\end{equation*}
where $(x,y)_z = \frac{1}{2}(d(z,x) + d(z,y) - d(x,y))$ for all $x,y,z \in X$.
A metric space which is $\delta-$hyperbolic for some $\delta \in [0, \infty)$ is said to be \textit{Gromov hyperbolic}.    
\end{defn}

\subsection{Asymptotic Cones}
Asymptotic cones were first introduced by Gromov in \cite{Gromov} formalises the notion of "looking at a metric space from infinitely far away" and its description in terms of ultrafilters were first given by van den Dries and Wilkies in \cite{DriesWilkies} . 

\begin{defn}
$\omega \in 2^\mathbb{N}$ is said to be an \textit{ultrafilter} if it is set of non-empty subsets of $\mathbb{N}$ which is closed under finite intersection, upwards closed and for any $X \subset \mathbb{N}$, either $X \in \omega$ or $\mathbb{N} \setminus X \in \omega$.
If no finite $X \subset \mathbb{N}$ belong to $\omega$, we say that $\omega$ is \textit{non-principal} .
\end{defn}

Ultrafilters are used to define \textit{ultralimits} which help us generalise the notion of limits to arbitrary sequences. 
If $(x_n)$ is a sequence of points in $(X,d)$, we say that $x \in X$ is an $\omega-$limit of $(x_n)$ ( denoted as $x = \lim_\omega x_n$ if for all $\epsilon > 0$, $\{ n : d(x_n, \epsilon) \leq \epsilon \} \in \omega$. 

We are now in a position to state the definition of asymptotic cones. 

\begin{defn}
Let $(X,d)$ be a metric space, $\mathrm{x} = (x_n)$ be a sequence of base points,  a sequence of positive numbers $(d_n)$ be a sequence of scaling factors and $\omega$ be an ultrafilter. 
Then by letting $X^\mathbb{N}_b$ being the set of sequences $(y_n) \subset M $ such that $\frac{d(x_n,y_n)}{d_n}$ is bounded. 

Define for sequences $\mathrm{a} = (a_n), \mathrm{b}= (b_n) \in X^\mathbb{N}_b$,  $d_\omega(\mathrm{a}, \mathrm{b})$ as $\lim_\omega \frac{d(a_n,b_n)}{d_n}$ with $\mathrm{a} \sim \mathrm{b}$ if and only if $d_\omega(\mathrm{a}, \mathrm{b) = 0}$.
Note that $d_\omega$ is a metric on $X^\mathbb{N}_b$ and we then call the resulting metric space $\mathrm{Con}_{\omega}(X, \mathrm{x}, (d_n))$ the asymptotic cone of $X$ with respect to $\omega, \mathrm{x}, (d_n)$.
\end{defn}

\subsection{Notation and Convention}
In this section, we shall clarify the notation that we shall be using throughout this article.
\begin{itemize}
     \item $d(x,y)$ is the distance between $x, y \in (M,g)$ or in a more general metric space.
     \item If $X$ is a metric space, then $\lambda X$ is $X$ with its metric being scaled by $\lambda$.
    \item $S(p,r) = \{ q \in M | d(p,q) = r \}$
    \item $\mathrm{int}(U)$ is the interior of a set 
    \item $l(\gamma)$ is the length of a curve $\gamma$ in $(M,g)$
    \item $(\mathbb{E}^2, g_{\mathbb{E}^2})$ is the Euclidean plane with the Euclidean metric.
    \item $\mathbb{H}^2$ is the hyperbolic plane and $g_{\mathbb{H}^2}$ its metric
    \item $\Delta$ is the Laplace-Beltrami operator associated to $(M,g)$
    \item $\mathrm{Con}_\omega(M, \mathrm{x}, (d_n))$ denotes an asymptotic cone of $(M,g)$ with $\mathrm{x} \in M^\mathbb{N}$ being a sequence of centres and $(d_n)$ being a sequence of scaling factors with $\lim_\omega d_n = \infty$.
    \item $2^X$ is the power set of $X$.
    \item  $\mathbb{D}^n$ is the closed unit n-disc
    \item $\overline{M}$ refers to the quasi-redirecting bordification of $M$. If $O \subset X$ is a subset , then $\overline{O}$ is its closure in $\overline{M}$.
    \item $\partial M$ is the quasi-redirecting boundary of $(M,g)$
    \item $\partial_\Delta M$ is the Martin boundary of $(M,g)$.
    \item $\overline{M}_\Delta$ is the Martin compactification of $(M,g)$.
\end{itemize}

\section{Planes with bounded geometry}
In this section, we shall study some Riemannian planes whose geometry is well-behaved with respect to curvature and injectivity radius.

\begin{defn}
 Let $(M,g)$ be a Riemannian manifold.
We say that $(M,g)$ is $(\kappa, \chi)-$bounded if there exists a $(\kappa, \chi) \in [0, \infty) \times (0, \infty)$ such that for every point $p \in M$,  $| \mathrm{sec}_p(\Pi) | \leq \kappa$ where $\Pi \subseteq T_pM$ is a 2-dimensional subspace and $\mathrm{inj}(p) > \chi$ for all $p \in M$.  

If a Riemannian manifold is $(\kappa, \chi)-$bounded for some $( \kappa, \chi) \in [0, \infty) \times ( 0, \infty)$, then we see that $(M,g)$ has \textit{bounded geometry}.
\end{defn}

From theorem IX.6.1 of \cite{Chavel}, it follows that $\mathrm{conv}(p) \geq \min \{ \frac{\mathrm{inj}(p)}{2}, \frac{\pi}{2 \sqrt{\kappa}} \}$ and therefore  
\begin{equation*}
    \mathrm{conv}(M) \geq \{ \frac{\mathrm{inj}(M)}{2}, \frac{\pi}{2\sqrt{\kappa}} \}
\end{equation*}
If $\sup_{p \in M}|\sec(p)|< \infty$, then $\mathrm{inj}(M) > 0$ is a necessary and sufficient condition for $\mathrm{conv}(M) > 0$.

It can be concluded that if $(M,g)$ has bounded sectional curvature and $\mathrm{inj}(M) > 0$, then $\mathrm{convx}(M) > 0$ and $(M,g)$ admits a bilipschitz triangulation $\mathcal{M}$ ( see theorems 1.1 and 1.2 of \cite{Bowditch2}) whose edges are all shorter than $\mathrm{convx}(M)$.

\begin{lemma}
\label{boundy}
Without loss of generality, we can assume that $\mathcal{M}$ has bounded degree.
\end{lemma}

\begin{proof}
Pick $p \in V(\mathcal{M})$ and consider $\mathrm{nbd}(v)$ to be those vertices $w \in V(\mathcal{M})$ for which $\{v, w\} \in E(\mathcal{M})$. 
As $-\kappa \leq\sec_p(\Pi) \leq \kappa$, we get from $\mathcal{M}$ being quasi-isometric to $(M,g)$ and from the Rauch comparison theorem that $|\mathrm{nbd}(v)| \leq k$ for some $k \in \mathbb{N}$. 
Therefore, we get that $\mathcal{M}$ is a planar graph with bounded degree and co-degree.
\end{proof}

\begin{theorem}
Let $(M,g)$ be a radially symmetric Riemannian plane where $g = dr^2 + f(r)^2 d\theta$.
If $| f''(r)/f(r)|$ is bounded and if $\liminf_{r \to \infty} f(r) > 0$, then there exists a $I > 0$ such that $\mathrm{inj}(q) \geq D$ for all $q \in M$.
\label{wellbehaved}
\end{theorem}

\begin{proof}
Recall that at a point $q =(r, \theta)$, $\mathrm{sec}(q) = - f''(r)/f(r)$.
If there exists a $C > 0$ such that $|\frac{f''(r)}{f(r)}| \leq C$ for all $r \geq 0$, it follows that $\mathrm{Conj}(q) \geq \frac{\pi}{\sqrt{C}}$ for all $q \in (M,g)$ ( see theorem 6.4.6 of \cite{Petersen}).

Let $\liminf_{r \to \infty} f(r) > 0$ and let $0 < D < \liminf_{r \to \infty}f(r)$.
There exists a $R > 0$ such that for all $t > R$, $f(t) \geq D$, .
Let $r < \frac{\pi}{4 \sqrt{C}}$.
From theorem 4.7 of \cite{Cheeger} we get that 
\begin{equation*}
\mathrm{inj}(q) \geq r \frac{\mathrm{vol}B(p,r)}{\mathrm{vol}B(p,r) + V_C(2r)}
\end{equation*}
where $V_C(2r)$ is the volume of a ball of radius $2r$ in a model space of constant sectional curvature $C$.
Note that $\mathrm{vol}B(q,r) \geq \mathrm{vol} \mathscr{Q}$ where $\mathscr{Q} = [ t-r/2, t+r/2] \times [\theta- r/2f(t), \theta + r/2f(t) ] \subset (M,g)$.
\begin{equation*}
\mathrm{vol}(\mathscr{Q}) = \int_{-r/2}^{r/2} f(t-x) \frac{r}{f(t)}dx \geq \frac{r}{D}. Dr= r^2
\end{equation*}
From this it follows that for all $r < \frac{\pi}{4 \sqrt V}$
\begin{equation*}
\mathrm{inj}(q) \geq r \frac{\mathrm{vol}B(p,r)}{\mathrm{vol}B(p,r) + V_C(2r)} \geq r \frac{r^2}{r^2 + V_C(2r)}
\end{equation*}
for all $q \in M$ for which $d(q,p) > R$.
Therefore $\mathrm{inj}(M) >0$ .
\end{proof}

\begin{theorem}
\label{radsim}
Let $(M,g)$ be a Riemannian plane with bounded geometry.
There exists a radially symmetric plane $(N,h)$ with bounded geometry which is quasi-isometric to $(M,g)$.
\end{theorem}

\begin{proof}
Let $\mathfrak{T}$ be a bounded degree bilipschitz triangulation of $(M,g)$ and let $\mathcal{T}$ be its $1-$skeleton.
We shall choose a root vertex 
\begin{equation*}
\label{fixpoint}
    v \in V(\mathcal{T})
\end{equation*}

Define $C_n$ to be $\partial B(v,n)$ for each $n \in \mathbb{N}$.
$\mathcal{T}$ being one-ended and planar ( as it is a triangulation of $\mathbb{R}^2$) forces all $C_n$ to be cycles.

Let $s_n = |C_n|$.
As $\mathcal{T}$ is a triangulation of $\mathbb{R}^2$, we see that $s_n \neq 1$ for all $n \in \mathbb{N}$.

Lemma \ref{boundy} shows that $\mathcal{T}$ has bounded degree so there exists a $\delta \in \mathbb{N}$ such that $\delta = \deg \mathcal{T}$.
This implies that
\begin{equation*}
  \frac{1}{\delta} \leq  \frac{s_n}{s_{n+1}} \leq \delta
\end{equation*}
for all $n \in \mathbb{N}$.

We may define $F : [1, \infty) \to [1, \infty)$ where 
\[
F(t) = \begin{cases}
    s_n \text{ if } t \in \mathbb{N} \\
     (s_{n+1} -s_n)t + s_n(1+n) -ns_{n+1} \text{ if } t\in [n, n+1] \text{ where } n \in \mathbb{N} 
\end{cases}
\]

Without loss of generality, we can take an $\epsilon$ and define a smooth $f: [0, \infty) \to [0, \infty)$  as follows
\[
f(t) = \begin{cases}
0 \text{ when } t = 0 \\
    s_n \text{ when } t \in \mathbb{N} \\
    F(t) \text{ when } t \notin [n - \epsilon, n+ \epsilon], n \in \mathbb{N}
\end{cases}
\]
We may further assume that
\begin{itemize}
    \item $f$ is strictly monotone on intervals $[n,n+1]$ where $n \in \mathbb{N} \cup \{0\}$
    \item There is a $0 \leq \kappa < \infty$ such that $\sup_{t \in (0, \infty)} |\frac{f''(t)}{f(t)} | \leq \kappa$.
    \item $f'(0) = 2\pi$
    \item $\lim_{r \to 0} \frac{f''(r)}{r} < \infty$ 
\end{itemize}
We see that there is a radially symmetric Riemannian plane $(N,h)$ with $h = dr^2 + \frac{f(r)^2}{4\pi^2} d\theta^2$.
From the way $f$ was constructed, we see that 
\begin{equation*}
    \liminf_{r \to \infty} \frac{f(r)}{2\pi r}= \liminf_{n \to \infty} \frac{s_n}{2 \pi } \geq \frac{1}{\pi}
\end{equation*}
This together with $|\frac{f''(t)}{f(t)}|$ being bounded ensures that $(N,h)$ exhibits bounded geometry ( see theorem \ref{wellbehaved} ).
Let $ 0 \leq \kappa < \infty$ be such that $|\sec(p)| \leq \kappa $ for all $p \in N$ and let $0 < \chi \leq \infty$ be such that $\mathrm{inj}(p) \geq \chi $ for all $p \in (M,g)$.
Theorem IX.6.1 in turn implies the existence of an $\eta > 0$ such that $\mathrm{conv}(p) \geq \eta$ for all $p \in (N,k)$.

Since $(N,h)$ is a radially symmetric plane, we may parametrise $(N,h)$ by radial co-ordinates where $(0,0)$ is identified with the pole $p \in (N,h)$. 

We may parametrise vertices $w \in V( \mathcal{T})$ as follows
\begin{equation}
    w= (\mathfrak{n},\mathfrak{i}) \text{ where } \mathfrak{n}=d_{\mathcal{T}}(v,w), \mathfrak{i} \in \{ 0,1,.., s_n - 1\}
\end{equation}
Note that the $v$ in the above equation is the same as the one in \eqref{fixpoint}.
We may label the vertices such that $(\mathfrak{n}, \mathfrak{i}), ( \mathfrak{n}, \mathfrak{j}) $ are adjacent if and only if $\mathfrak{i} - \mathfrak{j} = \pm 1 \mod{s_n}$.
We also see that if $( \mathfrak{n}, \mathfrak{i})$ and $(\mathfrak{m}, \mathfrak{j})$ are adjacent then $|\mathfrak{n}- \mathfrak{m} | \leq 1$.

Let $\mathfrak{C}_{n, n+1} \subset \mathcal{T}$ be a subgraph wherein $V(\mathfrak{C}_n) = C_n \cup C_{n+1}$ for all $n \in \mathbb{N} \cup \{0\} $ .
$e\in E(\mathfrak{C}_n)$ if and only if $e$ is an edge joining a vertex in $C_n$ and a point in $C_{n +1}$.
Being a subgraph of $\mathcal{T}$, $\mathfrak{C}_n$ is planar too.

This allows us to order $C_n \subset V(\mathcal{T})$ such for every $(\mathfrak{n}, \mathfrak{i}) \in V(\mathcal{T}) \setminus \{ v \}$, precisely one of the following holds 
\begin{itemize}
    \item $\zeta( \mathfrak{n}, \mathfrak{i}) = (n, \frac{2\pi i}{s_n} )$ and $(n \pm 1, \frac{2\pi i}{s_n}) = \zeta(\mathfrak{n} \pm 1, \mathfrak{j})$ for some $(\mathfrak{n} \pm 1, \mathfrak{j}) \in V(\mathcal{T})$
    \item $\zeta( \mathfrak{n}, \mathfrak{i}) = (n, \frac{2\pi i}{s_n})$ and $( n \pm 1, \frac{2\pi}{s_n})$ lies on an arc $\gamma \subset \partial B(p, n \pm 1)$ of length $1$ whose end points are $\zeta( \mathfrak{n} \pm 1, \mathfrak{j})$ and $\zeta( \mathfrak{n} \pm 1, \mathfrak{k})$.
    Furthermore, both $(\mathfrak{n}\pm 1, \mathfrak{j})$ and $(\mathfrak{n} \pm1 , \mathfrak{k})$ are adjacent to $( \mathfrak{n}, \mathfrak{i})$ and with each other.
\end{itemize}

 This taken together with $\deg \mathcal{T} = \delta$ and the triangle inequality implies that 
\begin{equation*}
    d_N(\zeta(\mathfrak{n}, \mathfrak{i}), \zeta(\mathfrak{m}, \mathfrak{j})) < \delta d_{\mathcal{T}}( (\mathfrak{n}, \mathfrak{j}), ( \mathfrak{m}, \mathfrak{j}))
\end{equation*}
holds true for all $( \mathfrak{m}, \mathfrak{j}), ( \mathfrak{n}, \mathfrak{i}) \in V(\mathcal{T})$.

Conversely, we observe from the way that $f$ was constructed that for any two adjacent vertices $( \mathfrak{n}, \mathfrak{i})$ and $( \mathfrak{m}, \mathfrak{j})$ that 
\begin{equation*}
    \frac{1}{\delta} d_{\mathcal{T}}(( \mathfrak{n}, \mathfrak{i}),( \mathfrak{m}, \mathfrak{j})) \leq d_N ( \zeta( \mathfrak{n}, \mathfrak{i}), \zeta( \mathfrak{m}, \mathfrak{j}))
\end{equation*}
Combining these equations and by taking the triangle inequality into account, we see that 
\begin{equation*}
    \frac{1}{\delta} d_{\mathcal{T}}(( \mathfrak{n}, \mathfrak{i}),( \mathfrak{m}, \mathfrak{j})) \leq d_N ( \zeta( \mathfrak{n}, \mathfrak{i}), \zeta( \mathfrak{m}, \mathfrak{j})) \leq  \delta d_{\mathcal{T}}( (\mathfrak{n}, \mathfrak{j}), ( \mathfrak{m}, \mathfrak{j}))
\end{equation*}
for all $( \mathfrak{n}, \mathfrak{i}), ( \mathfrak{m}, \mathfrak{j}) \in V( \mathcal{T})$.
Since $V(\mathcal{T})$ is quasi-isometric to both $(M,g)$ and $(N,h)$, this implies that $(M,g)$ and $(N,h)$ are quasi-isometric.
\end{proof}

\begin{cor}
Every Riemannian plane with bounded geometry $(M,g)$ can be written as $(N, \exp(2u) h) $ where $(N.h)$ is a radially symmetric Riemannian plane with bounded geometry which is quasi-isometric to $(M,g)$ and $u : N \to \mathbb{R}$ with $| \sup_{p \in M}| u(p) < \infty$
\end{cor}

\begin{proof}
Let $(M,g)$ be a Riemannian plane which exhibits bounded geometry.
Theorem \ref{radsim} guarantees the existence of a radially symmetric Riemannian plane $(N,h)$ which exhibits bounded geometry and which is quasi-isometric to $(M,g)$.

Theorem 1 of \cite{Kanai} taken together with the uniformatisation theorem tells us $(M,g)$ and $(N.h)$ are conformally equivalent.
Therefore there exists a $u : N  \to \mathbb{R}$ such that $(N, \exp(2u) h)$ is isometric to $(M,g)$.

For any path $\gamma : [0,1] \to N$, we shall denote its length in $N$ as $l_N( \gamma)$ and its length in $(M,g) = (N, \exp (2u) h)$ as $l_M ( \gamma)$.

Suppose that $\sup_{p \in N} | u(p) | = \infty$.
This implies that either $\inf_{p \in N} \exp(2u(p)) = 0$ or $\sup_{p \in N} \exp(2u(p)) = \infty$.
If it's the former, then for every $\epsilon > 0$, one may find  a point $p \in N$ and a $\delta_p < \mathrm{conv}(N)$ where $\sup_{q \in B(p, \delta_p)} \exp(2u(q)) < \epsilon$.
This implies that for all $\gamma : [0,1] \to B(p, \delta_p)$, we have that
\begin{equation*}
    l_M ( \gamma) < \epsilon l_N (\gamma)
\end{equation*}
This tells us that $(M,g)$ and $(N,h)$ are not quasi-isometric which is a contradiction.

Similarly, if $\sup_{p \in N}\exp(2u(p)) = \infty$, then for every $R > 0$, one may find a $p \in (N,h)$ and a $\delta_p < \mathrm{conv}(M)$ where $\inf_{q \in B(p,\delta_p)} \exp(2u(p)) > R$.
This implies that for all $\gamma : [0,1] \to B(p, \delta_p)$, we have that 
\begin{equation*}
  l_M   (\gamma) > R l_N (\gamma)
\end{equation*}
This too leads to a contradiction because it implies that $(M,g)$ and $(N,h)$ are not quasi-isometric.
\end{proof}

Now we shall show that the Gromov hyperbolic Riemannian planes $(M,g)$ are quasi-isometric to either $\mathbb{H}^2$ or $[0, \infty)$.

\textbf{Remark} - In the next theorem $\partial M$ shall be understood as referring to the Gromov boundary of $(M,g)$. 
\begin{theorem}
\label{hyp}
Let $(M,g)$ be a Riemannian plane with bounded geometry which is Gromov hyperbolic.
Either $(M,g)$ is conformally equivalent to $\mathbb{E}^2$ and is quasi-isometric to $[0, \infty)$ or $(M,g)$ is conformally equivalent and is quasi-isometric to $\mathbb{H}^2$.
\end{theorem}

\begin{proof}
From theorem 1.2 of \cite{Bowditch2} and from lemma \ref{boundy}, we see that $(M,g)$ admits a bilipschitz triangulation $\mathcal{M}$ with bounded degree which is quasi-isometric to $(M,g)$.

From theorem 11.3 of \cite{Bonk}, it follows that $\mathcal{M}$ is quasi-isometric to a convex subset of $\mathbb{H}^2$.

If $\partial M$ has a single point, then $(M,g)$ is quasi-isometric to $[0, \infty)$ and therefore, corollary 5.13 of \cite{KemperLokhamp} forces $(M,g)$ to be quasi-isometric to $[0, \infty)$.

If $\partial M$ is not a single point, then $\mathcal{M}$ is quasi-isometric to either $\mathscr{H}$- the closed half-plane in $\mathbb{H}^2$ or $\mathbb{H}^2$ itself ( see corollary 11.6 of \cite{Bonk}).

As $\partial \mathscr{H}$ is homeomorphic to $[0,1]$, corollary 5.13 of \cite{KemperLokhamp}  implies that $\partial_{\Delta} M$ is also homeomorphic to $[0,1]$ which contradicts the uniformatisation theorem.

We can therefore conclude that if $(M,g)$ is a Riemannian plane with bounded geometry, then $(M,g)$ is either conformally equivalent to $\mathbb{E}^2$ and quasi-isometric to $[0, \infty)$ or is conformally equivalent to $\mathbb{H}^2$ and quasi-isometric to $\mathbb{H}^2$.
\end{proof}

\section{Quasi-Redirection}
\label{sec: quasi-red}
In this section, we shall give a brief overview of quasi-redirection and the quasi-redirecting boundary which were introduced by Qing and Rafi to generalise the notion of Gromov hyperbolicity and the resulting compactification. 
For the rest of this section, we shall assume that all metric spaces are proper and geodesic. 

\begin{defn}
We say that a $(q_1,Q_1)-$quasigeodesic ray $\alpha$ \textit{eventually coincides} with a $(q_2, Q_2)-$quasigeodesic ray $\beta$ if there exists $a, b , T $ such that for all $t \geq T$, $\alpha(t) = \beta(at+b) $.
\end{defn}

\begin{defn}
Let $\alpha, \beta : [ 0 , \infty) \to X$ be two quasi-geodesic rays in $X$.
We say that $\alpha $ can be \textit{quasi-redirected} to $\beta$ ( denoted by $\alpha \preceq \beta$ ) if there exists a pair of constants $(q,Q)$ with $q \geq 1$ and $Q \geq 0$ such that for every $r > 0$, there exists a $(q,Q)-$quasi-geodesic ray $\gamma_r$ that is identical to $\alpha$ inside of $B(\alpha(0),r)$ and that $\gamma_r$ eventually coincides with $\beta$ outside of $B( \alpha(0),r)$.
If $\alpha \preceq \beta$ and $\beta \preceq \alpha$, then we say that $\alpha \sim \beta$.
\end{defn}

Denote by $P(X)$ the resulting set of equivalence classes of quasi-geodesic rays which is topologised in a particular manner provided $X$ satisfies these three assumptions listed below-

$\textbf{Assumption 0}$- $X$ is a proper, geodesic metric space. Furthermore, there exists a pair of constants $(q_0, Q_0)$ such that every $x \in X$ lies on an infinite $(q_0, Q_0)-$quasi-geodesic ray .

$\textbf{Assumption 1}$-  For $(q_0, Q_0)$ as in Assumption 0, every equivalence class of quasi-geodesics $\mathfrak{a} \in P(X)$ contain s a $(q_0,Q_0)-$quasi-geodesic ray. Fix such a $(q_0,Q_0)$-ray and call it the central element $\underline{a}$ of $\mathfrak{a}$ .

$\textbf{Assumption 2}-$ For every $\mathfrak{al} \in P(X)$, there is a function $f_\mathfrak{a} : [1, \infty) \times [0, \infty) \to [1, \infty) \times [0, \infty)$ such that if $\beta \prec \alpha $, then any $(q,Q)-$ray $\beta \in \mathfrak{b}$ can be $f_\mathfrak{a}((q,Q))-$quasi-redirected to $\underline{a}$ .

$\textbf{Assumption 3}-$If $\alpha \prec \beta$, then $\beta \prec \alpha$.

\textbf{Remark}- We say that quasi-redirection on a metric space $(X,d)$ is symmetric if it satisfies Assumption 3.

\textbf{Topologising the quasi-redirecting boundary}- For each $x \in X$, define $\mathbf {x} = \{$ quasi-geodesic rays passing through $x \}$ and as points in $P(X)$ are equivalence classes of quasi-geodesic rays, we topologise $\overline{X} =X \cup P(X)$ by letting $\mathfrak{a} \in P(X)$,

we define $\mathcal{U}(\mathfrak{a}, r) = \{ \mathfrak{b} \in P(X) \cup X |$ every $(q,Q)-$ray in $\mathfrak{b}$ can be $F_\mathfrak{a}((q,Q))-$ redirected to $\underline{a}$ at radius $r$ . 

Sets of the form $\mathcal{U}( \mathfrak{a}, r)$ along with open subsets of $X$ define a topology on $X \cup P(X) = \overline{X}$.
We shall denote $P(X)$ with the induced topology ( as a subspace of $\overline{X}$) as $\partial X$.

If $(X,d)$ is a bounded metric space, then we get that $\partial X = \phi$ as $(X,d)$ has no quasi-geodesic rays.
If $(X,d)$ is Gromov hyperbolic, then we get from \cite{QingRafi} that the quasi-redirecting boundary is the Gromov boundary and the quasi-redirecting compactification is the Gromov compactification.

\begin{defn}
We shall say that a metric space $(X,d)$ is monodirectional if $P(X)$ has precisely one element. 
It is obvious that for such a space, $\partial X$ is also a single point.
\end{defn}

\textbf{Remark}- If a proper geodesic metric space is $(X,d)$ is monodirectional, we shall define a topology on $X \cup P(X)$ to be the one-point compactification of $X$ even though $(X,d)$ may not satisfy all the quasi-redirecting assumptions.

\begin{lemma}
Let $(X,d)$ be a geodesic metric space which satisfies Assumptions 0-2.
Then $\overline{X}$ is Hausdorff if and only if $X$ satisfies Assumption 3.
\end{lemma}

\begin{proof}
Suppose that $\alpha \prec \beta$ and $\beta \nprec \alpha$.
Let $\mathrm{a} \in \partial X$ be the redirecting class of $\alpha$ and let $\mathrm{b} \in \partial X$ be the redirecting class of $\beta$.

We see that $\mathrm{a} \in \mathcal{U}( \mathrm{b}, r)$ for all $r > 0$ but there exists $R > 0$ such that for all $r \geq R$, $\mathrm{b} \notin \mathcal{U}( \mathrm{a}, r)$ which ensures that $\overline{X}$ is not Hausdorff.

If quasi-redirection is symmetric, then $\mathrm{a} \in \mathcal{U}( \mathrm{b},r)$ for all $r > 0$ implies that $\mathrm{a} \prec \mathrm{b}$.
Symmetry of quasi-redirection should force $\mathrm{a} = \mathrm{b}$.
Therefore, $\partial M$ is Hausdorff.
$X \subset \overline{X}$ being open in turn forces $\overline{X}$ to be Hausdorff if $\partial X$ is Hausdorff.
\end{proof}

Recall that theorem 5.9 of \cite{QingRafi} states that if $(X,d)$ is a geodesic metric space which satisfies the quasi-redirecting assumptions 0-2 and $(Y,d')$ is a metric space quasi-isometric to $X$, then the quasi-isometric mapping induces a homeomorphism between $\partial X$ and $\partial Y$.

We can relax the quasi-redirecting assumptions if $X$ is monodirectional and both $X$ and $Y$ are homeomorphic

\begin{cor}
Let $(X,d)$ and $(Y, d')$ be two proper geodesic metric spaces and let $\Psi : X \to Y$ be a quasi-isometric homeomorphism.
If $(X,d)$ is monodirectional, then so is $(Y,d')$.
\end{cor}

\begin{proof}
Let $\alpha, \beta : [0, \infty) \to (X,d)$ are respectively  $(q_1,Q_1)-$ and $(q_2, Q_2)-$ quasigeodesic rays. 
If $\alpha \prec \beta$, there exists a $(q_3, Q_3) \in [1, \infty) \times [0, \infty)$ and a family of $(q_3, Q_3)$ quasi-geodesic rays $\Theta_r$ such that $\Theta_r$ coincides with $\alpha$ in $B( \alpha(0), r)$ and $\Theta_r$ eventually coincides with $\beta$.

As $\Psi : (X,d) \to (Y,d')$ is a quasi-isometric homeomorphism, we see that for any $(q,Q) \in [1, \infty) \times [0, \infty)$, there exists a $(\Psi(q), \Psi(Q)) \in [1,\infty) \times [0, \infty)$ such that if $\gamma : [0, \infty) \to (X,d)$ is a $(q,Q)$-quasigeodesic ray, then $\Psi(\gamma)$ is a $(\Psi(q), \Psi(Q))-$quasigeodesic ray.

By denoting $\Psi(r) = \inf_{p \in \partial B( \alpha(0),r)} d'( \Psi( \alpha(0), \Psi(p))$, we see that $\Psi(\alpha)$ $(\Psi(q_3), \Psi(Q_3))-$redirects to $\Psi( \beta)$ for all $r > 0$.
By taking $\alpha, \beta : [0, \infty) \to (Y,d)$, we get that $\Psi^{-1}(\alpha) \sim \Psi^{-1}(\beta)$ ( as $(X,d)$ is monodirectional).
This ensures that $(Y,d')$ is also monodirectional.
\end{proof}

\subsection{Quasi-redirection in Riemannian planes}
As quasi-redirection is invariant under quasi-isometries, we can just work with radially symmetric planes without losing any generality and understanding the divergence of geodesics is a crucial tool in understanding quasi-redirection in planes.

In this section, we shall study radially symmetric planes $(M,g)$ which have a pole at $p \in M$ and which has a metric $g = dr^2 + f(r)^2d\theta^2$ .

\begin{theorem}
\label{basepoints}
Let $\alpha : [0, \infty) \to (M,g)$ be a smooth $(q,Q)$-quasigeodesic ray. 
Suppose that $\alpha(0) \neq p$.
Then we can construct a smooth quasi-geodesic ray $\beta : [0, \infty) \to M$ where $\beta(0) =p$ and $\beta(t) \in \alpha$ for all $t \geq T$ for some $T > 0$.
Furthermore, $\alpha \sim \beta$.
\end{theorem}

\begin{proof}
Let $l= d( \alpha(0), p)$ and let $\epsilon < \mathrm{conv}(\alpha(0))$ and let $B(\alpha(0), \epsilon$ be a geodesic ball centred at $\alpha(0)$.
Let $T= \sup_{t >0 } \{ t : d(\alpha(0), d(\alpha(t)) = \epsilon \}$.
Denote as $s \in \partial B(\alpha(0),\epsilon)$ the point which is nearest to $p$ and let $\gamma \subset B(\alpha(0), \epsilon)$ be a segment joining $[p,s]$ and $\alpha|_{[T,\infty)}$ such that $\beta = [p,s] \cup \gamma \cup \alpha|_{[T, \infty)}$ is a smooth ray.
We may further assume that $l(\gamma) \leq 2 \epsilon$.
Denote the geodesic ray starting from $p$ and passing through $q$ as $\eta$.

We may parametrise $\gamma : [0, l(\gamma) ] \to M$ by letting $\gamma(0) =s$, $\gamma(l(\gamma)) = \alpha(T)$ and $l(\gamma([0,t]) = t$.
We may parametrise $\beta$ as follows
\[
\beta(t) = \begin{cases}
    \eta(t) &\text{ if } t \in [0,d(\alpha(0)),p) - \epsilon] \\
    \gamma(t- d(\alpha(0),p)+ \epsilon) &\text{ if } t \in [d(
    \alpha(0),p) - \epsilon, d(\alpha(0),p) - \epsilon + l(\gamma) ]\\
    \alpha(t + T + \epsilon - d( \alpha(0),p) -l(\gamma)  &\text{ if } t \in [d(\alpha(0), p) -\epsilon + l(\gamma) ,\infty)
\end{cases}
\]
If $t_1, t_2 \in [0, d(\alpha(0),p) -\epsilon]$, then $d(\beta(t_1), \beta(t_2)) = |t_2 -t_1|$.
If $t_1, t_2 \in [d(\alpha(0), p) - \epsilon + l(\gamma), \infty)$, then 
\begin{equation*}
    \frac{1}{q}|t_1-t_2| -Q \leq d(\beta(t_1), \beta(t_2)) \leq q|t_2-t_1|+Q
\end{equation*}
If $t_1 \in [0, d(\alpha(0),p) - \epsilon] $ and $t_2 \in [d(\alpha(0),p) -\epsilon, d(\alpha(0), p)-\epsilon + l(\gamma) ]$, then we see that 
\begin{equation*}
 t_2 -t_1 -l(\gamma) \leq   d(\beta(t_1), \beta(t_2)) \leq t_2 -t-1
\end{equation*}
If $t_1 \in [d(\alpha(0),p) -\epsilon, d( \alpha(0), p) -\epsilon + l(\gamma)]$ and $t_2 \in [d(\alpha(0),p)- \epsilon+ l(\gamma), \infty)$, we see that 
\begin{equation*}
    \frac{1}{q}( t_2 -t_1) - Q -l(\gamma) \leq d(\beta(t), \beta(t_2)) \leq q(t_2 -t_1) +Q
\end{equation*}
If $t_1 \in [0, d(\alpha(0),p) -\epsilon)]$ and $t_2 \in [ d(\alpha(0),p) -\epsilon + l(\gamma), \infty)$, we see that
\begin{equation*}
    \frac{1}{q}(t_2 -t_1) -Q - l(\gamma) \leq d( \beta(t_1), \beta(t_2)) \leq q(t_2 -t_1) +Q
\end{equation*}
Combining the above equations with $l ( \gamma) \leq 3 \epsilon$, we get that 
\begin{equation*}
    \frac{1}{q}(t_2 -t_1) -Q - 2 \epsilon \leq d( \beta(t_1), \beta(t_2)) \leq q(t_2 -t_1) + Q + 2 \epsilon
\end{equation*}
This ensures that $\beta$ is a $(q,Q+2 \epsilon)-$quasigeodesic ray. 
As $\alpha$ is a $(q,Q)-$quasigeodesic ray, $\alpha$ is a $(q, Q+ 2\epsilon)-$quasigeodesic ray too.

To show that $\alpha \sim \beta$, it suffices to show that there exists a $C > 0$ such that $d(\alpha(t), \beta(t)) \leq C$ for all $t > 0$.

If $t \in [0, d(\alpha(0),p) - \epsilon]$, we see that 
\begin{equation*}
    d( \alpha(t), \beta(t)) \leq q d( \alpha(0), p) + Q + \epsilon
\end{equation*}
If $t\in [d(\alpha(0), p) - \epsilon, d(\alpha(0),p) + l(\gamma) - \epsilon]$, then 
\begin{equation*}
    d( \alpha(t), \beta(t)) \leq  q\epsilon + qd(\alpha(0),p) + Q
\end{equation*}
If $t \in [d(\alpha(0), p) + l(\gamma) - \epsilon, \infty) $, then 
\begin{equation*}
    d( \alpha(t), \beta(t)) \leq qd(\alpha(0),p) + q\epsilon + 2 Q
\end{equation*}
Combining these three and by taking a small enough $\epsilon$, we get that $d( \alpha(t), \beta(t)) \leq q ( d(\alpha(0),p) + \epsilon) +2 Q$.
Lemma 2.7 of \cite{QingRafi} ensure that $\alpha$ and $\beta$ $(q, q(d(\alpha(0),p) + \epsilon) + 3Q)$ redirect to each other at all $r >0$.
Therefore, $\alpha \sim \beta$.
\end{proof}

We shall now check the case where quasi-geodesic rays intersect each other infinitely many often.

\begin{theorem}
\label{twins}
If $\alpha, \beta : [0, \infty ) \to M$ are quasi-geodesic rays which intersect each other infinitely many times, then $\alpha \sim \beta$.
\end{theorem}

\begin{proof}
Suppose that $\alpha$ is a $(q_1, Q_1)-$quasigeodesic ray and $\beta$ is a $(q_2, Q_2)-$quasigeodesic ray.
We do not lose any generality by assuming that both $\alpha$ and $\beta$ are smooth.
Theorem \ref{basepoints} allows us to assume that $\alpha (0) = \beta(0) = p$.

For any $r > 0$, we may find $t_1, t_2$ such that $\alpha(t_1) = \beta(t_2)$, $d(\alpha(t_1),p) > r$, $d( \alpha(t), p) > r$ for all $t > t_1$ and $d(\beta(t), p) > r$ for all $t > t_2$.
For the sake of convenience we shall call $\alpha(t_1) = \beta(t_2)$ as $q$.

We may further assume that $t_1 = \inf_{t > 0} \{ t: \alpha(t) = q \}$ and $t_2 = \sup_{t > 0} \{ t : \beta(t) = q\}$.

As $(M,g)$ is a Riemannian manifold, there is an $\epsilon >  0$ such that the geodesic ball $B(q, \epsilon)$ is convex.
By taking $\epsilon$ small enough, we can ensure that $d( B(q, \epsilon), p) > r$.

Let $T_1, T_2$ be such that $d(\alpha(T_1), q) = \epsilon$ and $d(\beta(T_2), q) = \epsilon$.
We may further assume that $T_1 = \inf_{t > 0} \{ t: d(\alpha(t), q) = \epsilon \}$ and that $T_2 = \sup_{t > 0} \{ t: d( \beta(t), q) = \epsilon \}$.

Let $\varphi$ be a smooth path joining $\alpha(T_1)$ and $\beta(T_2)$ lying entirely in $B(q, \epsilon)$ such that the concatenated path $\gamma = \alpha|_{[0, T_1]} \cup \varphi \cup \beta|_{[T_2, \infty)} $ is a smooth embedding of $[0, \infty)$ in $(M,g)$.
Denote the length of $\varphi$ as $l(\varphi)$.
Parametrise $\varphi : [0, l(\varphi) ] \to M$ where $\varphi(0) = \alpha(T_1)$, $\varphi(l(\varphi)) = \beta(T_2)$ and $l(\varphi([0,t])) = t$ whenever $t \in [0, l(\varphi(t))]$.

We may define a ray $\gamma_r : [0, \infty) \to M$ where 
\[
\gamma_r(t) = \begin{cases}
    \alpha(t) &\text{if } t \in [0, T_1] \\
    \varphi( t-T_1) &\text{if } t \in [T_1, T_1 + l(\varphi)] \\
    \beta (t+ T_2- T_1-l(\varphi)) &\text{if } t \in [T_1 + l(\varphi), \infty) 
\end{cases}
\]
We see that if $s_1, s_2 \in [0,T_1]$, then $ \frac{1}{q_1} |s_1-s_2|- Q \leq d(\gamma_r(s_1), \gamma_r(s_2)) \leq q_1 |s_2 -s_1| + Q_1$.
If $s_1, s_2 \in [T_2, \infty)$, then $\frac{1}{q_2} |s_1 -s_2| -Q_2 \leq d( \gamma_r( s_1) , \gamma_r(s_2)) \leq q_2|s_1-s_2|+Q_2$.

If $s_1 \in [0, T_1]$ and $t_2 \in [s_1, l(\varphi)]$, we see that 
\begin{equation*}
 \frac{1}{q_1} (s_2-s_1) -Q_1 -l(\varphi)  \leq d(\gamma_r(s_1) , \gamma_r(s_2)) \leq q_1(s_2 -s_1) + Q_1 
\end{equation*}
Similarly if $s_1 \in [T_1, T_1 + l(\varphi)] $ and $s_2 \in [T_2, \infty) $, we see that 
\begin{equation*}
 \frac{1}{q_2}(s_2 - s_1) - Q_2 -l ( \varphi) \leq  d( \gamma_r(s_1), \gamma_r(s_2)) \leq q_2( s_2  - s_1) + Q_2 
\end{equation*}
If $s_1 \in [0, T_1]$ and $s_2 \in [T_2, \infty)$, then we see that 
\begin{equation*}
\frac{1}{q_1q_2} (s_2 -s_1) -Q_1-Q_2 -l(\varphi) \leq   d( \gamma_1(s_1), \gamma_2(s_2) ) \leq q_1q_2(s_2 -s_1) + Q_1 +Q_2
\end{equation*}
This ensures that $\gamma_r$ is a $(q_1q_2, Q_1+Q_2+ l(\varphi))$-quasigeodesic ray which coincides with $\alpha$ in $B(\alpha(0),r)$ and eventually coincides with $\beta $.
Our choice of $r$ being arbitrary and $l(\varphi) < \epsilon$ ensures that $\alpha \prec \beta$.
Similarly we establish that $\beta \prec \alpha$.
Therefore $\alpha \sim \beta$.
\end{proof}

We shall use this to characterise those Riemannian planes which are monodirectional.

\begin{lemma}
\label{mono}
Let $(M,g)$ be a radially symmetric Riemannian plane with $g = dr^2 + f(r)^2 d\theta^2$ where $\liminf_{r \to \infty} \frac{f(r)}{r} < \infty$.
Then $(M,g)$ is monodirectional.
\end{lemma}

\begin{proof}
Suppose that $\liminf_{r \to \infty} \frac{l(f(r))}{r} = \lambda$ for some $\lambda \in [0, \infty)$.

Therefore we can find for each $R >0$ and $\epsilon > 0$, a $r>R$ such that $f(r) < (\lambda + \epsilon)r $.
It follows that there exists a sequence $(r_n)$ such that $\lim_{n \to \infty} r_n = \infty$ such that $f(r_n) < (\lambda + \epsilon)r_n$ for all $n\in \mathbb{N}$.

Now consider a geodesic $\alpha : [0, \infty) \to M $ with $\alpha(0) = p$ and let $-\alpha : [0, \infty) \to M$ be a geodesic ray where $-\alpha(t) = \exp_p( -\exp_p^{-1}(\alpha(t)))$.
Denote by $\Gamma_r$ the shortest path joining $\alpha(r)$ and $-\alpha(r)$ lying entirely in $\overline{M \setminus B(p, r)}$,
We shall now construct a quasi-geodesic ray $\gamma : [0, \infty) \to (M,g)$ where 
\[
\gamma(t) = \begin{cases}
    \alpha(t) &\text{for } t \in [0, r_1] \\
    \Gamma_{r_1}(t -r_1) &\text{for } t \in [r_1, 2r_1] \\
    -\alpha(t - r_1) &\text{for } t \in [r_1, r_2] \\
    \Gamma_{r_2}(2r_2 -t) &\text{for } t \in [r_2, 2r_2] \\
    ......
\end{cases}
\]

It is obvious from construction that $\gamma$ is a $(\lambda+ \epsilon, 0)-$quasigeodesic ray whose limit points in $UT_pM$ is an arc $\mathfrak{a}$ with endpoints $\exp_p^{-1}(\alpha(1))$ and $\exp_p^{-1}(-\alpha(1))$.

Denote as $\eta$ the quasi-geodesic ray where $\eta(0) = p$ and $\eta (t) = \exp_p( - \exp_p^{-1}(\gamma(t)) $.
Notice that $\alpha$ and $-\alpha$ intersect $\gamma$ and $\eta$ infinitely many times which implies that $\alpha \sim -\alpha \sim \gamma \eta$ ( see theorem \ref{twins}).
We note that every quasi-geodesic ray $\zeta : [0, \infty) \to M$ intersect either $\gamma$ or $\eta$ infinitely many times which ensures that for any two quasi-geodesic rays $\alpha, \beta$, we have that $\alpha \sim \beta$.

Therefore $(M,g)$ is monodirectional.
\end{proof}

From theorem 1 of \cite{Kanai}, we see that $(M,g)$ and $(N,h)$ being Riemannian surfaces exhibiting bounded geometry which are quasi-isometric are also conformally equivalent.
Therefore, we do not lose any generality by limiting ourselves to radially symmetric planes exhibiting bounded geometry if we intend to study properties which are invariant under quasi-isometries. 

 \begin{theorem}
 \label{focal}
Let $(M,g)$ be a radially symmetric plane with bounded geometry where $g = dr^2 + f(r)^2 d\theta^2$.
Further assume that $\lim_{r \to \infty} \frac{f(r)}{r} = \infty$
There exists a radially symmetric plane $(N,h)$ with bounded geometry which is quasi-isometric to $(M,g)$ and satisfies the following properties
\begin{itemize}
    \item $h =dr^2 + F(r)^2 d\theta^2 $ and $\lim_{r \to \infty} \frac{F(r)}{r} = \infty$
    \item $F : [0, \infty) \to [0, \infty)$ is strictly increasing.
\end{itemize}
 \end{theorem}

\begin{proof}
Let $(M,g)$ be a radially symmetric plane with bounded geometry where $ g =dr^2 +f(r)^2 d\theta^2$ and $\lim_{r \to \infty} \frac{f(r)}{r} = \infty$.

From L'Hospital's rule it follows that $\lim_{r \to \infty} f'(r)  = \infty$ which implies that for all $t> 0$, there exists a $R_t$ such that for all $r > R_t$, $f'(r) > t$.

Let $R > 0$ be such that $f'(r) \geq 1$ for all $r \geq R$.
We can therefore extend $f_{[R, \infty)}$ to a smooth function $F$ such that $F = f$ on $[R, \infty)$ and $F'(r) > 0$ for all $r \in [0, R]$ which in turn makes $F$ strictly increasing.

Let $(N,h)$ be a Riemannian plane with $h = dr^2 + F(r)^2 d\theta^2$ with its pole at $p$.
$\liminf_{r \to \infty} F(r) = \infty$ together with $F(r) =f(r)$ whenever $r \in [R, \infty)$ ensures that $(N,h)$ exhibits bounded geometry. ( see theorem \ref{wellbehaved})

Let  $ N' = \{ x \in N | d(x,p) \geq R\}$ and $M' =\{ x \in M | d(x,p) \geq R\}$ and let $(M',d_M)$ and $(N',d_N)$ be the metrics on $M' , N'$ induced from $(M,g)$ and $(N,h)$.
As $F \equiv f$ on $[R, \infty)$, it follows that $M'$ and $N'$ are isometric.
Define $\iota : M \to N$ to be the mapping $\iota ((r, \theta)) = (r,\theta)$.
Choose $q,s \in (M,g)$ and let $\gamma : [0,1] \to M$ be a path with $\gamma(0) =q$ and $\gamma(1) =s$.
We may parametrise $\gamma(t) = ( \gamma_r(t), \gamma_\theta(t))$
Recall that  
\begin{equation*}
    l_M(\gamma) = \int_0^1 \sqrt{\gamma_r'(t)^2 + f(\gamma_r(t))^2 \gamma_\theta'(t)^2 }dt
\end{equation*}
and 
\begin{equation*}
    l_N (\gamma) = \int_0^1 \sqrt{\gamma_r'(t)^2 + F(\gamma_r(t))^2 \gamma_\theta'(t)^2 }dt
\end{equation*}
As $F(r) = f(r)$ whenever $r \geq R$, there exists a $C_1, C_2 > 0$ such that $C_1F(r) \geq f(r)$ and $C_2F(r) \leq f(r)$ for all $r \in (0, \infty)$.
This implies that 
\begin{equation*}
C_1 l_N(\gamma) \leq  l_M(\gamma) \leq C_2l_N(\gamma)
\end{equation*}
Therefore we can find a radially symmetric plane $(N,h)$ with $h = dr^2 + F(r)^2 d\theta^2$ which is quasi-isometric to $(M,g)$, exhibits bounded geometry and has $F$ as strictly monotone.
\end{proof}

If $(M,g)$ is a Riemannian plane exhibiting bounded geometry which has no proper asymptotic cones, then $(M,g)$ is quasi-isometric to a radially symmetric plane with bounded geometry wherein $g = dr^2 +f(r)^2d\theta^2 $ and $f'(r) > 0$ for all $r \geq 0$.

\begin{theorem}
\label{convexballs}
Let $(M,g)$ be a radially symmetric plane with its pole at $p$ and $g = dr^2 + f(r)^2 d\theta^2$ where $f$ is strictly monotone on $[0, \infty)$.
Then the geodesic balls $B(p,r)$ are convex for all $r > 0$.
\end{theorem}

\begin{proof}
Let $\gamma: [0, \infty) \to M$ be a geodesic ray with $\gamma(0) = p$. 
It suffices to show that any Jacobi field $J$ on $\gamma$ with $J(0) =0$ is strictly increasing ( see \cite{Sullivan}).

If $p = (r, \theta) \in M$, then $\mathrm{sec(p)}= - \frac{f''(r)}{f(r)}$ for all $p \in M$ and denote by $K(r)$ the function $-\frac{f''(r)}{f(r)}$.
Now let $\gamma : (-\infty, \infty) \to M$ be a geodesic with $\gamma(0) = p$.
We can define a basis of $T_{\gamma(t)} M$ by letting $E_1 = \frac{\Dot{\gamma}(t)}{g_{\gamma(t)}(\Dot{\gamma}(t), \Dot{\gamma}(t))}$ and $ E_2 \in T_{\gamma(t)}M$ be such that $g_{\gamma(t)}(E_2,E_2) = 1$ and $g_{\gamma(t)}(E_1, E_2) = 0$.

Let $J$ be a Jacobi field along $\gamma$ which is parallel to $\gamma$. 
We note that $J(t) = A \Dot{\gamma}(t) + Bt\Dot{\gamma}(t)$ where $A, B \in \mathbb{R}$.
If $J$ is a Jacobi field which is orthogonal to $\gamma$, then the Jacobi equation reduces to $J'' + K(r)J = 0$ where $K(r) = -\frac{f''(r)}{f(r)}$ is the sectional curvature at $q = (r, \theta)$.

We note that Jacobi fields orthogonal to $\gamma$ are of the form $J(t) = A f(t) + B\frac{f(t)}{\int_0^t f(s)^2ds}$.
Therefore any Jacobi field $J$ along $\gamma$ is of the form -
\begin{equation*}
    J(t) = (A+Bt)E_1 + (Cf(t) + D \frac{f(t)}{\int_0^t f(s)^2 ds} )
\end{equation*}

It follows that if $J(0) = 0$, then $J = At E_1 + Bf(t) E_2$ for some $A, B \in \mathbb{R}$ and we also note that $||J(t_1) || > || J(t_2) ||$ if and only if $t_1 > t_2$.
Therefore, geodesic balls centred at $p$ are convex.
\end{proof}

From \cite{Goto}, it follows that $\partial_\infty M$ is homeomorphic to $\mathbb{S}^1$ and that for every $\zeta \in \partial_\infty M$, there exists a unique $v \in UT_pM$ such that $\lim_{t \to \infty} b_{v,t} = \zeta$. 
Let $x_n = \exp_p (t_n v_n)$ where $t_n \in [0, \infty)$ and $v_n \in UT_pM $.
We note that $\lim_{n \to \infty}x_n \in \partial_{\infty} M$ if and only if $\lim_{n \to \infty} t_n = \infty$ and $\lim_{n \to \infty} v_n = v \in UT_p M$.

\begin{lemma}
\label{uniquerays}
Let $(M,g)$ be a radially symmetric plane with its pole at $p$  and $\lim_{r \to \infty} \frac{f(r)}{r} = \infty$.
Let $\alpha, \beta : [0, \infty) \to M$ be geodesic rays with $\alpha(0) = \beta(0) = p$.
Then $\alpha \prec \beta$ if and only if $\alpha = \beta$.
\end{lemma}

\begin{proof}
Without loss of generality, we can assume from lemma \ref{focal} that $f$ is strictly monotone.
Let $\alpha  ,\beta : [0, \infty ) \to M$ be two geodesic rays with $\alpha(0) = \beta(0) = p$ and $\alpha \neq \beta$.
We note that from convexity of geodesic balls $B(p,r)$ ( see theorem \ref{convexballs} ) that the length of the shortest arc in $\partial B(p,r)$ joining $\alpha(r)$ and $\beta(r)$ is  $|\theta| f(r)$ where $\cos \theta = g_p ( \exp^{-1}_p (\alpha(1)), \exp_p^{-1}(\beta(1)))$ and $\theta \in [-\pi, \pi]$.

Assume that $\alpha \prec \beta$.
For all $r > 0$, there exists a $(Q,q)$-quasigeodesic ray $\Theta_r$ which coincides with $\alpha$ in $B(p, r)$ and which for some $T_r >r $ , $\Theta_r(t) \in \beta$ for all $t \geq T_r$.

Without loss of generality, we can assume that there exists an $a_r \geq r$ such that $\Theta_r(t) = \alpha(t)$ for all $t \in [0, a_r]$ and a $b_r \geq r$ such that for some $T > 0$, we have that $\Theta_r(t+ T) = \beta( b_r +t )$ for all $t > 0$.

We may further assume that $\Theta_r(t) \notin \alpha$ for all $t > a_r$ and $\Theta_r(t) \notin \beta$ for all $t < T$.
Denote as $\gamma_r \subset \Theta_r$ the quasi-geodesic segment joining $\alpha(a_r)$ and $\beta(b_r)$ which lies in $\Theta_r$ and denote its length as $l(\gamma_r)$.

We may choose a $\gamma_r$ such that $\alpha|_{[0, a_r]} \cup \gamma_r \cup \beta_{[b_r, \infty)} $ is a smooth embedding of $[0, \infty)$ into $(M,g)$.

Parametrise $\gamma_r : [0, l(\gamma)] \to M$ where $\gamma_r(0)= \alpha(a_r)$,  $\gamma_r(l(\gamma_r)) = \beta( b_r)$ and $l( \gamma_r([0,t])) = t$ where $t \in [0, l(\gamma_r)]$.

We may parametrise $\Theta_r$ as follows
\[
\Theta_r(t) = \begin{cases}
    \alpha(t) \space  \text{if} \space t \in [0, a_r] \\
    \gamma_r(t-a_r) \space  \text{if} \space t \in [a_r, a_r + l(\gamma_r)] \\
    \beta( t+b_r -a_r -l(\gamma_r)) \space \text{if} \space t \in [a_r + l(\gamma_r) , \infty)
\end{cases}
\]
If $t_1, t_2 \in [0, a_r]$ or $t_1, t_2 \in [ a_r + l(\gamma) , \infty)$, then $d( \Theta_r(t_1), \Theta_r(t_2)) = |t_1 -t_2|$.

If $t_1 \in [0, a_r]$ but $t_2 \in [a_r, l(\gamma)]$, then 
\begin{equation*}
  ( t_2 -t_1) - l(\gamma_r) -\epsilon  \leq d(\Theta_r(t_1), \Theta_r(t_2)) \leq t_2 -t_1
\end{equation*}

If $t_1 \in [a_r, l(\gamma_r)]$ and $t_2 \in [a_r+l(\gamma_r), \infty)$, then
\begin{equation*}
 t_2 - t_1  -l(\gamma_r)  \leq  d( \Theta_r(t_1), \Theta_r(t_2)) \leq t_2 - t_1
\end{equation*}
If $t_1 \in [0, a_r] $ and $t_2 \in [ a_r + l(\gamma_r), \infty) $, then
\begin{equation*}
 t_2 - t_1 -l(\gamma_r)     \leq d( \Theta_r(t_1), \Theta_r(t_2)) \leq t_2 -t_1  
\end{equation*}
Combining these, we see that 
\begin{equation*}
 t_2 -t_1  - l(\gamma_r)   \leq d ( \Theta_r( t_1), \Theta_r(t_2)) \leq t_2-t_1  + l(\gamma_r) 
\end{equation*}
This ensures that $\Theta_r$ is a $(1, l(\gamma_r))-$quasigeodesic ray. 

Theorem \ref{convexballs} implies that $l(\gamma_r) \geq |\theta| f(r)$ where $\theta \in [-\pi, \pi]$ and 
\begin{equation*}
    \cos \theta = g_p( \exp_p^{-1}(\alpha(1)), \exp_p^{-1}(\beta(1)))
\end{equation*}
As $l(\gamma_r) \geq |\theta|f(r) $, this ensures that $\lim_{r \to \infty} \frac{l(\gamma_r)}{r} = \infty$.
This ensures that there is no $(q,Q)$ such that $\alpha$ $(q,Q)-$redirects to $\beta$ at all $r > 0$.
Therefore if $\alpha, \beta$ are two geodesic rays with $\alpha(0) = \beta(0) = p$, then $\alpha \prec \beta$ if and only if $\alpha = \beta$.
\end{proof}

\begin{cor}
\label{noidle}
Let $(M,g)$ be a radially symmetric plane with bounded geometry which has its pole at $p\in M$.
Let $\alpha : [0, \infty) \to M$ be a $(q,Q)-$quasigeodesic ray with $\alpha(t) = p$ if and only if $t= 0$.
If $\mathcal{L} =\{ \frac{\exp^{-1}_p(\alpha(t))}{d(\alpha(t),p)} \}_{t > 0} \subset UT_pM$ has more than one limit point, then $(M,g)$ is monodirectional.
\end{cor}

\begin{proof}
Without loss of generality, we can assume that $\alpha (0) = p$ and $\alpha(t) \neq p$ for all $t > 0$.
As $\mathcal{L} =\{ \frac{\exp^{-1}_p(\alpha(t))}{d(\alpha(t),p)}\}_{t>0} \subset UT_pM$ has more than one point, it has an arc $\gamma \subset \mathcal{L}$.
Choose $v_1, v_2 \in \gamma$ such that neither of which is an endpoint of $\gamma$.
There exists geodesic rays $\gamma_1, \gamma_2 : [0, \infty) \to M$ with $\gamma_i(t) = \exp_p(tv_i)$ which intersect $\alpha$ at infinitely many points.

Theorem \ref{twins} imply that $\gamma_1 \sim \alpha$ and $\gamma_2 \sim \beta$ from which it follows that $\gamma_1 \sim \gamma_2$.
\end{proof}

Fix the pole $p \in (M,g)$ as the base-point with respect to which Busemann functions are defined. 
For each $q \in M, b_q : M \to \mathbb{R}$ such that $b_q(z) = d(z, q) - d(p,q) $.
We note that if $\gamma : [0, \infty) \to M$ with $\gamma(0) = p$, then $\lim_{t \to \infty} b_{\gamma(t)}$ exists and we denote it as $b_\gamma$.

\begin{prop}
\label{idle}
Let $(M,g)$ be a radially symmetric plane with its pole at $p$ and let $g = dr^2 + f(r)^2 d\theta^2$ where $\lim_{r \to \infty} \frac{f(r)}{r} = \infty$.
Further assume that $(M,g)$ exhibits bounded geometry.
Then quasi-redirection is symmetric in $(M,g)$.
\end{prop}

\begin{proof}
Theorem \ref{basepoints} allows us to assume that $\alpha (0) = p$ and theorem \ref{focal} allows us to assume that $f$ is strictly monotone.
Theorem \ref{convexballs} and \cite{Goto} ensures that for every $\zeta \in \partial_\infty M$, there is a unique geodesic ray $\gamma : [0, \infty) \to M$ with $\gamma(0) =p$ for which $\lim_{t \to \infty} \gamma(t) = \zeta $. 

Let $\alpha : [0, \infty) \to M$ be a $(q,Q)-$quasigeodesic ray.
As  $\lim_{r \to \infty} \frac{f(r)}{r} = \infty$, we see that there is a unique geodesic ray $\gamma$ such that $\lim_{t \to \infty}[p, \alpha(t)] = \gamma$ ( see corollary \ref{noidle}).

$\alpha : [0, \infty) \to M$ is a $(q,Q)-$quasi-geodesic implies that for all $r \geq 0$, $d(\alpha(t), p) \geq R$ whenever $t \geq Q(r+ q)$.

It follows from \cite{Goto} that $\lim_{t \to \infty} b_{\alpha(t)} = b_{\gamma}$.
As $\lim_{t \to \infty} b_{\alpha(t)} = b_\gamma$, we get that 
\begin{equation*}
    \lim_{t \to \infty} d(\alpha(t),z) -d(\alpha(t), p) -d(\gamma(t),z) + d(\gamma(t), p) = 0
\end{equation*}
Thus $\lim_{t \to \infty} d(\alpha(t),z) - d(\alpha(t), p) - d(\gamma(t),z) + t = 0$ for all $z \in M$.

Let $t_r = \sup \{ t > 0 | d(\alpha(t), p) = r\}$. 
As $\alpha$ is a $(q,Q)$-quasigeodesic, it follows that $t_r \in [\frac{r-Q}{q},q(r+Q)]$ for all $r > 0$.

For a given $k \geq 0$, it follows that for all $\epsilon > 0$, there exists $R_\epsilon > 0$ such that 
\begin{equation*}
   d(\gamma(r), \alpha(t_{(1+k)r})) - kr < \epsilon 
\end{equation*}
for all $r > R_\epsilon$.
This implies that $\limsup_{r \to \infty} \frac{d( \gamma(r), \alpha(t_{(1+k)r})}{r} \leq k$.
By taking $k \to 0$, it follows that 
\begin{equation*}
    \lim_{r \to \infty} \frac{d(\alpha(t_r), \gamma(r))}{r} = 0
\end{equation*}
This implies the existence of an $R > 0$ such that $d(\alpha(t_r), \gamma(r)) \leq r/2$ for all $r> R$.
Lemma 2.6 of \cite{QingRafi} ensures that $\alpha$ $(9q,Q)-$redirects to $\gamma$ and thus, $\alpha \prec \gamma$.

For the converse, we see that the geodesic segments $[p, \alpha(t)]$ converge uniformly on compact subsets of $(M,g)$ to $\gamma$ which ensures that $\gamma$ redirects to $(3q, Q+1)-$redirects to $\alpha$ at all $r> 0$.
Therefore, $\alpha \sim \gamma$ and every redirecting class has  $\mathrm{a}$ has a unique geodesic ray $\alpha_0$ which starts from $p$.
Lemma \ref{uniquerays} ensures that geodesic rays $\alpha, \beta : [0, \infty) \to M$ with $\alpha(0) = \beta(0) = p$ never redirect to each other except when $\alpha = \beta$.

Therefore, quasi-redirection is symmetric in $(M,g)$.
\end{proof}

\begin{cor}
\label{assumptions}
Let $(M,g)$ be a Riemannian plane with bounded geometry which is not monodirectional.
Then $(M,g)$ satisfies  the quasi-redirecting assumptions 0-3.
\end{cor}

\begin{proof}
From theorem \ref{radsim} and theorem 5.9 of \cite{QingRafi}, we can assume that $(M,g)$ is a radially symmetric plane with its pole at $p$ and from lemma \ref{focal}, we can assume that $g = dr^2 + f(r)^2 d\theta^2$ where $\lim_{r \to \infty} \frac{f(r)}{r} = \infty$ and $f$ is strictly monotone.

Proposition \ref{idle} ensures that quasi-redirection is symmetric.
Therefore, $(M,g)$ satisfies Assumption 3.
$(M,g)$ being a complete Riemannian manifold where for all points $q \in M$, one can find a geodesic ray passing through $q$ ensures that $(M,g)$ satisfies assumption 0 ( take $\mathrm{q} = (1,0)$ in Assumption 0).

Proposition \ref{idle} ensures that for every redirecting class $\mathrm{a}$, we may choose $\alpha \in \mathrm{a}$ where $\alpha_0 : [0,\infty) \to M$ as its central element.
Therefore $(M,g)$ satisfies assumption 1.

Proposition \ref{idle} also ensures that any quasi-geodesic ray $\alpha \in \mathrm{a}$ $(9q,Q)-$redirects to $\alpha_0$ ( the central element of $\mathrm{a}$). 
Therefore $(M,g)$ satisfies Assumption 2.
 
\end{proof}

Therefore, we can topologise $M = \overline{M} \cup \partial M$ as in \cite{QingRafi}.

\subsection{Topologising the boundary for non-monodirectional planes}
In this section, we shall assume that $(M,g)$ is a radially symmetric plane with bounded geometry where $g = dr^2 + f(r)^2 d\theta^2$ and 
\begin{itemize}
    \item $\lim_{r \to \infty} \frac{f(r)}{r} = \infty$
    \item $f$ is strictly monotone
\end{itemize}
We note a similarity in the proofs of this section and that of section 8 of \cite{QingRafi} which dealt with geodesic metric spaces which were asymptotically tree-graded with respect to monodirectional subsets. 

\begin{lemma}
\label{upperbounds}
Let $\mathrm{a} \in \partial M$ and $r>0$ be given and let $\alpha_0 \in \mathrm{a}$ be its central element.
We may choose $\mathrm{f}_{\mathrm{a}}$ large enough such that the following holds.
Namely that there exists a $r'>r$ such that if a geodesic ray $\beta_0 $ can be $\mathrm{f}_{\mathrm{a}}((1,0))-$redirected to $\alpha_0$ at $r'$ and $\beta_0 \in \mathrm{b}$ is its central element, then $\mathrm{b} \in \mathcal{U}( \mathrm{a}, r)$.
\end{lemma}

\begin{proof}
Since $(M,g)$ is radially symmetric, we may parametrise $\alpha_0(t) =(t, \theta_{\mathrm{a}})$ where $\theta \in [0, 2\pi)$.
Define $\mathrm{f}_{\mathrm{a}} (q,Q)= (9q+1, q+Q)$ for all $\mathrm{a} \in \partial M$ and for all $(q,Q) \in [1, \infty) \times [0, \infty)$.
Theorem \ref{basepoints} can allow us to restrict ourselves to working with quasi-geodesic rays $\beta : [0, \infty) \to M$  where $\beta(0) = p$.
Lemma 4.3 of \cite{submorse} and proposition \ref{idle} allows us to take $r'= 3r$ as it guarantees that any $(q,Q)-$ray $\beta$ $(9q,Q)-$redirects to $\beta_0$ ( $\beta_0 \in [\beta]$ being the central element of its redirecting class) via a $(9q,Q)-$ray $\zeta$ which coincides with $\beta$ in $B(p,r)$ and which coincides with $\beta_0$ outside of $B(p, 3r)$.

Take $\beta_0 : M \to [0, \infty)$ where $\beta_0(0) = p$ and let its redirecting class be $\mathrm{b}$.
Define $\theta =  \min \{ |\theta_{\mathrm{a}}-\theta_{\mathrm{b}}|, 2\pi -|\theta_{\mathrm{a}}-\theta_{\mathrm{b}}| \}$.
We see that if $\theta < \frac{1}{f(3r)}$, then $\beta_0$ $(1,\theta f(r) + \epsilon)-$redirects to $\alpha_0$ at radius $3r$ where $\epsilon < \mathrm{conv}(M)$.
If we take $\epsilon$ small enough, we see that $\beta_0$ $(1,1)-$redirects to $\alpha_0$ at radius $3r$.
By using an argument similar to the proof of Lemma 3.2 of \cite{QingRafi}, we can show that every $\beta \in \mathrm{b} = [\beta_0]$ $(9q,Q)-$redirects to $\alpha_0$ at radius $r$.

Therefore, if a geodesic ray $\beta_0 : [0, \infty) \to M$ with $\beta_0(0) =p$ can be $\mathrm{f}_{\mathrm{a}}((1,0))-$redirected to $\alpha_0$ at radius $3r$, then $\mathrm{b} = [\beta_0] \in \mathcal{U}( \mathrm{a}, r)$.
\end{proof}

\begin{prop}
\label{secondcount}
$\partial M$ is second countable.
\end{prop}

\begin{proof}
Since $(M,g)$ exhibits bounded geometry and is not monodirectional, we may assume that $(M,g)$ is a radially symmetric plane with $g = dr^2 + f(r)^2 d\theta^2$ where $f$ is strictly monotone and $\lim_{r \to \infty} \frac{f(r)}{r} = \infty$.
Let $\epsilon < \frac{1}{2} \min\{ 1, \chi\}$ where $\chi$ is the convexity radius of $(M,g)$.

$(M,g)$ being radially symmetric means that for any geodesic ray $\gamma : [0, \infty) \to M$ with $\gamma(0) = p$ can be written as $\gamma(t) = ( t, \theta_\gamma)$. 

Let $(r_n)$ be a strictly monotone sequence where $\lim_{n \to \infty} r_n = \infty$.
$(M,g)$ being proper ensures that there are finitely many points which are $1-$dense in $\partial B(p,r_n)$.
As every point in $(M,g)$ lie on a geodesic ray starting from $p$, there exists a finite set $\mathcal{E}_n$ of geodesic rays $\gamma$ which start from $p$ such that the set 
\begin{equation*}
    \{ \gamma(r_n) : \gamma \in \mathcal{E}_n \}
\end{equation*}
is $1-$dense in $\partial B(p,r_n)$.
It is plain to see that $\mathcal{D} \subset UT_pM$ is dense where 
\begin{equation*}
    \mathcal{D} = \{ \exp^{-1}(\alpha(1)) \in UT_pM : \alpha \in \bigcup_n \mathcal{E}_n \}
\end{equation*}

Define 
\begin{equation*}
\mathcal{Y}= \{ \mathcal{U}( [\gamma], R) :\gamma \in \bigcup_n \mathcal{E}_n, R \in \mathbb{N} \}
\end{equation*}
We see that $\mathcal{Y}$ is a countable collection of open subsets of $\partial M$.
It suffices to show that for every $\mathrm{b} \in \partial M$ and every $r > 0$, there is a $\mathcal{U}( \mathrm{a}, R) \in \mathcal{Y}$ such that 
\begin{equation*}
\mathrm{b} \in \mathcal{U}( \mathrm{a}, R) \subset \mathcal{U}( \mathrm{b}, r)
\end{equation*}
For every redirecting class $\mathrm{b}$, its central element $\beta_0 $ is a geodesic ray of the form $\beta_0(t) = (t, \theta_{\mathrm{b}})$.
$\mathcal{D}$ being dense in $UT_pM$ allows us to choose a $\mathcal{U}( \mathrm{a}, R) \in \mathcal{Y}$ where $R > \max \{f^{-1}(2f(3r)), 3r \}$ and 
\begin{equation*}
    \min \{ |\theta_{\mathrm{a}}-\theta_{\mathrm{b}}|, 2\pi -|\theta_{\mathrm{a}}-\theta_{\mathrm{b}}| \} < \frac{1}{f(3R)} 
\end{equation*}
Lemma \ref{upperbounds} implies that $\mathrm{b} \in \mathcal{U}( \mathrm{a}, R)$.
Let $\mathrm{c} \in \mathcal{U}( \mathrm{a}, R)$ and let $\gamma_0$ be its central element.
We see that this implies that 
\begin{equation*}
  \min \{ | \theta_{\mathrm{c}} - \theta_{\mathrm{b}}| , 2\pi - |\theta_{\mathrm{c}} -\theta_{\mathrm{b}}|\} < \frac{1}{f(R)}
\end{equation*}
We see that this implies that 
\begin{equation*}
\min \{ |\theta_{\mathrm{a}} - \theta_{\mathrm{c}}| , 2\pi - |\theta_{\mathrm{a}} - \theta_{\mathrm{c}}|\} < \frac{2}{f(R)} \leq \frac{1}{f(3r)}
\end{equation*}
Lemma \ref{upperbounds} implies that $\mathrm{c} \in \mathcal{U}( \mathrm{a}, r)$. 
As $\mathrm{c} \in \mathcal{U}( \mathrm{a},R)$ was arbitrary; it follows that
\begin{equation*}
   \mathrm{b} \in \mathcal{U}(\mathrm{a},R) \subset  \mathcal{U}(\mathrm{b}, r)
\end{equation*}
Thus $\partial M$ is second countable.
\end{proof}

Now we shall establish that $\partial M$ is regular.
\begin{prop}
\label{reggie}
$\partial M$ is regular. 
\end{prop}

\begin{proof}
Let $\mathrm{a} \in \partial M$ and let $B \subset \partial M$ be a closed subset.
$B$ being closed implies the existence of an $r>0$ for which 
\begin{equation*}
B \cap \mathcal{U}(\mathrm{a}, r) = \phi   
\end{equation*}
For a given $r > 0$, take $r' = 3r$ as in lemma \ref{upperbounds}.
We will show that for every $\mathrm{b} \in B$, there is a $r_{\mathrm{b}}$ such that 
\begin{equation*}
    \mathcal{U}( \mathrm{a}, r') \cap \mathcal{U}( \mathrm{b}, r_{\mathrm{b}}) = \phi 
\end{equation*}
Then $\bigcup_{\mathrm{b} \in B} \mathcal{U}( \mathrm{b}, r_{\mathrm{b}})$ is an open neighbourhood of $B$ which is disjoint from $\mathcal{U}( \mathrm{a}, r')$.

Pick $\mathrm{b} \in B$ and let $\beta_0 \in \mathrm{b}$ be its central element and assume for contradiction that 
\begin{equation*}
    \mathcal{U}( \mathrm{a}, r') \cap \mathcal{U}( \mathrm{b}, R) \neq \phi
\end{equation*}
for all $R> 0$.
Then, there is a sequence $\mathrm{c}_i \in \mathcal{U}( \mathrm{a}, r') \cap \mathcal{U}( \mathrm{b}, R_n)$ where $\lim_{n \to \infty} R_n = \infty$.
Let $\gamma_n \in \mathrm{c}_n$ be their central elements. 
Upto taking a subsequence, $\gamma_n$ converges to a geodesic ray $\gamma : [0, \infty) \to M$ with $\gamma(0) = p$.
We see that $\gamma$ is the central element of $\mathrm{c}$ which is its redirecting class.
It follows that $\gamma$ can be quasi-redirected to $\beta_0$ at all $R> 0$ and therefore $\gamma \in \mathrm{b}$ and $\mathrm{b} \in \mathcal{U}( \mathrm{a}, r') \subset \mathcal{U}( \mathrm{a}, r)$ which is a contradiction.
\end{proof}

\begin{theorem}
\label{circular}
$\partial M$ and $\mathbb{S}^1$ are homeomorphic.
\end{theorem}

\begin{proof}
We showed in proposition \ref{secondcount} that $\partial M$ is second countable and in proposition \ref{reggie} that $\partial M$ is regular.
Corollary \ref{assumptions} ensures that $(M,g)$ satisfies assumptions 0-3 and lemma 7.3 of \cite{QingRafi} ensures that $\partial M$ is Hausdorff.
Uryshon's metrisable theorem ensures that $\partial M$ is metrisable.

Choose a sequence $(\mathrm{a}_n) \subset \partial M$. 
For each $\mathrm{a}_n$, let  $\alpha^n_0$ be its central element.
We may parametrise $\alpha^n_0 : [0, \infty) \to M$ as $\alpha^n_0(t) = f(t, \theta_n)$.
Let $T = (\theta_n)$ where the $\theta_n$ are the same ones in the previous sentence.
$UT_pM$ being compact and metrisable ensures that there exists $( \theta_{n_k}) \subset (\theta)$ such that $\lim_{k \to \infty} \theta_{n_k} = \theta \in UT_pM$.

Let $\beta_0 : [0,\infty) \to M$ be the geodesic ray where $\beta(t) = (t, \theta)$.
We see that $\beta_0$ is the central element of its redirecting class which we denote as $\mathrm{b}$.

We intend to show that $\lim_{k \to \infty}\mathrm{a}_{n_k} = \mathrm{b}$.
We see that for each $r > 0$, if $ \min \{|\theta_{n_k}- \theta |, 2\pi - | \theta_{n_k} - \theta| \} \leq \frac{1}{f(3r)}$, then $\mathrm{a}_{n_k} \in \mathcal{U}( \mathrm{b}, r)$ which ensures that $\lim_{k \to \infty} \mathrm{a}_{n_k} = \mathrm{b}$
Therefore $\partial M$ is compact.

To see that $\partial M$ and $\mathbb{S}^1$ are homeomorphic, recall that $\mathbb{S}^1$ is homeomorphic to $UT_pM$.
Proposition \ref{idle} and lemma \ref{uniquerays} ensures that we can parametrise every $\mathrm{a}$ by $\theta_{\mathrm{a}}$ where $\alpha_0 :[0, \infty) \to M$ with $\alpha_0(t) = (t, \theta_0)$ is its central element.
We see that $\partial M \mapsto UT_pM$ by $\mathrm{a } \mapsto \theta_{\mathrm{a}}$ is a continuous bijection.
Both the aforementioned spaces being compact and metrisable together with $UT_pM$ being homeomorphic to $\mathbb{S}^1$ ensures that $\partial M$ is homeomorphic to $\mathbb{S}^1$.
\end{proof}

\section{Comparison with Martin Boundary}
\textbf{Martin boundary}- Let $(M,g)$ be a complete Riemannian manifold and let $\Delta_g$ be the Laplace-Beltrami operator whose Green's function is $\mathscr{G}$.

The vector space $\mathcal{H}(M)$ of harmonic functions on $(M,g)$ with semi-norms
\begin{equation*}
    ||u||_K = \sup_K |u(x)|_{x \in K} , K \subseteq M \text{ is compact}
\end{equation*}
is a Frechet space.
Fix $o \in M$ and let $\mathcal{K}_o = \{ u \in \mathcal{H}_+(M) : u(o) = 1\}$ where $\mathcal{H}_+(M) \subset \mathcal{H}(M)$ is the set of positive harmonic functions. 
We can define the Martin kernel
\begin{equation*}
    k(x,y) = \frac{\mathscr{G}(x,y)}{\mathscr{G}(o,y)}
\end{equation*}
A sequence $(y_n) \subset M$ is called a Martin sequence if $\lim_{n \to \infty} d(y_n, o) = \infty $ and $\lim_{n \to \infty} k(x, y_n) = u(x)$ where $u$ is harmonic.
Deem two Martin sequences equivalent if they converge to the same harmonic function.
The collection of all such sequences is called the Martin boundary and is denoted as $\partial_\Delta M$. 
We see that $\partial_\Delta(M) \subseteq \mathcal{K}_o$ is compact.
$\overline{M}_\Delta$ is the Martin compactification of $(M,g)$. ( see \cite{Wang} ) 

\textbf{Remark}- If $(M,g)$ is a closed manifold, we shall take $\partial_\Delta M = \Phi$ and $\overline{M}_\Delta$ as $(M,g)$ has no sequence $(q_n)$ which has $\lim_{n \to \infty} d(q_n, p) = 0$ for any $p \in M$.
If $(M,g)$ is non-compact and $\Delta$ does not admit a Green's function, then $\partial_\Delta M$ is a single point and $\overline{M}$ is the one-point compactification of $(M,g)$.

\begin{theorem}
There exists a Riemannian plane $(M,g)$ with bounded geometry which is conformally equivalent to the Euclidean plane and which is not monodirectional.  
\end{theorem}

\begin{proof}
Let $(M,g)$ be radially symmetric with its pole at $p$ and $g =dr^2 + f(r)^2 d\theta$ such that $f(r) = r \log r$ for all $r \geq 2$.
\begin{equation*}
    \int_2^\infty \frac{dr}{r \log r} = \infty
\end{equation*}
Proposition 3.1 of \cite{Grigoryan}implies that $(M,g)$ is conformally equivalent to $\mathbb{E}^2$.
We note that $\liminf_{r \to \infty} f(r) = \infty$ and for all $q \in M : d(p,q) \geq 2$, we have 
\begin{equation*}
    \mathrm{sec}(q) = - \frac{1}{d(p,q)^2 \log d(p,q)}
\end{equation*}
These imply that $(M,g)$ has bounded geometry.

As $\lim_{r \to \infty} \frac{f(r)}{r} = \infty$, it follows from theorem \ref{circular}, lemma \ref{uniquerays} and proposition \ref{idle} that $\partial M = \mathbb{S}^1$.
\end{proof}

We shall also see that a Riemannian plane with bounded geometry  which is conformally equivalent to $\mathbb{H}^2$ can be monodirectional.

\begin{theorem}
There exists a Riemannian plane with bounded geometry $(M,g)$ which is conformally equivalent to $\mathbb{H}^2$ and is also monodirectional.
\end{theorem}

\begin{proof}
Let $\chi : \mathbb{R} \to [0,1]$ be a smooth function such that
\[
\chi(r) = \begin{cases}
    0 &\text{if} \space |r| \geq \frac{1}{2} \\
    \exp(\frac{1}{r^2 - 1/4}) \exp(-4) &\text{if} \space |r| \leq \frac{1}{2}
\end{cases}
\]
We note that $\chi$ is smooth and that $\chi(0) = 1$ and $\chi'(0) = 0$.
Define a smooth function $f : [1, \infty) \to [1, \infty)$ such that 
\[
f(r) = \begin{cases}
    r^2 - (n^6 -n^3)\chi (\frac{r -n^3}{n^\frac{3}{2}}) &\text{if} \space r\in [ n^3 - \frac{1}{2n^{\frac{3}{2}}}, n^3 + \frac{1}{2n^{\frac{3}{2}}}], n \in \mathbb{N}, n \geq 2 \\
    r^2 &\text{otherwise}
\end{cases}
\]
Let $F : [0, \infty) \to [0, \infty)$ be a smooth function such that $F(t) > 0$ whenever $t > 0$, $F(0) = 0, F'(0) =1$ and $F(t) = f(t)$ whenever $t \geq 1$.
There exists a Riemannian plane $(M,g)$ which has its pole at $p$ and which has the metric $g = dr^2 + F(r)^2 d\theta^2$.

We note that $\int_1^\infty \frac{dr}{F(r)}= \int_1^\infty \frac{dr}{r^2}+ \sum_{n =2}^\infty n^{-3/2} < \infty$ which ensures that $(M,g)$ is conformally equivalent to $\mathbb{H}^2$ ( see proposition 3.1 of \cite{Grigoryan}).

We also note that whenever $r \notin [n^3-\frac{1}{2}n^{\frac{3}{2}}, n^3 + \frac{1}{2}n^{\frac{3}{2}}]$, $|\frac{f'"(r)}{f(r)}| = \frac{2}{r^2}\leq 2$.
As $\chi'' : \mathbb{R} \to \mathbb{R}$ is bounded on account of $\chi$ having compact support, it follows that there exists a $D > 0$ such that whenever $r \in [ n^3-\frac{1}{2}n^{-3/2}, n^3 + \frac{1}{2}n^{-3/2}]$ and $n \geq 2, n \in \mathbb{N}$ 
\begin{equation*}
   |F''(r)/F(r)|= |\frac{2- (n^3-1)\chi(s)(\frac{9/2s^2 -1/8}{(s^2-1/4)^4})}{r^2 -(n^6-n^3) \chi(s) }| \leq D  
\end{equation*}
where $s = \frac{r-n^3}{n^{3/2}}$.
This implies that $(M,g)$ has bounded curvature.

We note that whenever $r \geq 1$ and $r \neq [n^3 - \frac{1}{2}n^{-3/2}, n^3 + \frac{1}{2}n^{-3/2}], n \geq 2, n \in \mathbb{N}$, it follows that $\frac{F(r)}{r}= r$, therefore $\limsup_{r \to \infty} \frac{F(r)}{r} = \infty$.
If $r \in [n^3 - \frac{1}{2}n^{-3/2}, n^3 + \frac{1}{2}n^{-3/2}]$, then $\inf \frac{F(r)}{r} = 1$ which is attained when $r = n^3$.
Therefore, $\liminf \frac{F(r)}{r} = 1$ which ensures that $(M,g)$ is monodirectional .
As $\liminf F(r) = \infty$, we get that $\mathrm{inj}(M) > 0$ and therefore $(M,g)$ has bounded geometry.
\end{proof}

\section{Asymptotic Cones}

In this section we shall study asymptotic cones of Riemannian planes $(M,g)$ with bounded geometry.

As asymptotic cones are preserved under homeomorphisms by quasi-isometries, proposition 2.5 of \cite{Riley} and theorem \ref{radsim} allows us to assume without losing generality that $(M,g)$ is a radially symmetric plane that has bounded geometry.

\begin{defn}
All asymptotic cones of radially symmetric $(M,g)$ in this paper will be defined with respect with $\mathrm{p} = (p_n)$ where $p_n = p$ for all $n \in \mathbb{N}$.
We shall call $\mathrm{p}$ to be the base point.
\end{defn}

We see that we may parametrise the asymptotic cones of radially symmetric Riemannian planes by using " radial co-ordinates" too.

\begin{lemma}
\label{coord}
Let $(M,g)$ be a radially symmetric plane, let $\omega$ be a nonprincipal ultrafilter,
and let $\lim_\omega r_n = \infty$.
Let $\mathrm{Con}_\omega(M,\mathrm{p},(r_n))$ be an asymptotic cone of $(M,g)$.

If $\mathrm a=((a_n,\theta_n))\in \mathrm{Con}_\omega(M,\mathrm p,(r_n))$ such that $d_\omega(\mathrm a,\mathrm p)=D,$ then $d_\omega(\mathrm a,\mathrm b)=0$ where $\mathrm((Dr_n,\theta_n))$.
\end{lemma}

\begin{proof}
Since $\mathrm p$ is the origin in polar coordinates, we have $d((a_n,\theta_n),\mathrm p)=a_n$.
Thus $D=d_\omega(\mathrm a,\mathrm p)=\lim_\omega \frac{d((a_n,\theta_n),\mathrm p)}{r_n}=\lim_\omega \frac{a_n}{r_n}$.
Therefore we see that $d_\omega(\mathrm a,\mathrm b)=\lim_\omega \frac{d((a_n,\theta_n),(Dr_n,\theta_n))}{r_n}=\lim_\omega \frac{|Dr_n-a_n|}{r_n}=\lim_\omega \left|D-\frac{a_n}{r_n}\right|=0$.
\end{proof}

This ensures that every $\mathrm{a} \in \mathrm{Con}_\omega (M, \mathrm{p},(d_n))$ has a representative of the form $( (d_\omega (\mathrm{a}, \mathrm{p})d_n, \theta_n)) \in M^\mathbb{N}$.

\textbf{Remark}\label{rem}-If $\mathrm{q} = (q_n) = ((r_n, \theta_n))\in \mathrm{Con}_\omega(M, \mathrm{p}, (r_n))$, we define $r_n= d(p,q_n)$ and $\theta_n$ as follows
\[
\theta_n = \begin{cases}
    (\exp_p^{-1}q_n)/d(q_n,p) \text{ if } d(q_n,p) \neq 0 \\
    0 &\text{otherwise}
\end{cases}
\]

By arguing along the lines of Kapovich and Kleiner in theorem 8.1 of \cite{lacunary}, we can indeed show that "lacunary hyperbolic" planes whose asymptotic cone is a non-proper $\mathbb{R}-$tree for some non-principal ultrafilter $\omega$ are indeed Gromov-hyperbolic.

\begin{theorem}
The following statements are equivalent.
\begin{itemize}
    \item $(M,g)$ is Gromov hyperbolic and quasi-isometric to $\mathbb{H}^2$.
    \item Every asymptotic cone of $(M,g)$ is a non-proper $\mathbb{R}-$tree
    \item There is a non-principal ultrafilter $\omega$ and a sequence of positive numbers $(r_n)$ with $\lim_\omega r_n = \infty$ such that $\mathrm{Con}_\omega (M, (r_n))$ is a non-proper $\mathbb{R}-$tree.
\end{itemize}
\end{theorem}

\begin{proof}
The first statement implies the second statement and the second statement implies the third statement so it suffices to show that the third statement implies the first statement.

Let $(M,g)$ be a Riemannian plane which is " lacunary hyperbolic", that is to say, $(M,g)$ has an asymptotic cone which is a non-proper $\mathbb{R}-$tree.

Since $(M,g)$ has bounded geometry, theorem 1.2 of \cite{Bowditch2} implies that $(M,g)$ admits a bi-Lipschitz triangulation, we see that $(M,g)$ is $m-$simply connected, in the sense of definition 8.1 of \cite{Bowditch}.
Since $(M,g)$ is $m-$simply connected, it follows from theorem 8.1.2 of \cite{Bowditch} that the following holds -

If for all $k$, there exists a $k'$ and there exists a $R$ such that for every $p \in (M,g)$, $B(p,R)$ being $k-$hyperbolic implies that $\mathcal{T}$ is $k'-$hyperbolic.

Assume that $\mathrm{Con}_\omega (M, \mathrm{x}, (r_n)) $ is a $\mathbb{R}-$tree which is not proper and where $\mathrm{x} \in M^\mathbb{N}$ is arbitrary.

$\mathrm{Con}_\omega (M, \mathrm{x}, (r_n))$ being a non-proper $\mathbb{R}-$tree implies the existence of $(\delta_j)$ such that 
\begin{equation*}
    \lim_\omega \frac{\delta_j}{r_j} = 0
\end{equation*}
and every ball $B(p, r_j)$ with $p \in (M,g)$ is $\delta_j-$hyperbolic.

$\lim_\omega r_j = \infty$ implies that for every $n \in \mathbb{N}$ and for every $\lambda > 0$, there is a $j > n$ such that $r_j > R$ and $\frac{\delta_j}{r_j}< \lambda$.
This implies that each $B(p, r_j)$ is $\lambda r_j-$hyperbolic which ensures that each $B(p, R)$ is $\lambda r_j$-hyperbolic.

Theorem 8.1.2 of \cite{Bowditch} implies that there exists a $\kappa$ such that $(M,g)$ is $\kappa-$hyperbolic.
Theorem \ref{hyp} implies that $(M,g)$ is quasi-isometric to $\mathbb{H}^2$.
\end{proof}

However, there are many Riemannian planes with bounded geometry which are not hyperbolic and yet admit non-proper asymptotic cones $\mathrm{Con}_\omega (M, (r_n))$ for some non-principal ultrafilter $\omega$ and some sequence $(r_n) \in (0, \infty)^\mathbb{N}$ with $\lim_\omega r_n = \infty$ ( as we shall see in theorem \ref{spread}).

\begin{defn}
Let $(X,d)$ be a geodesic metric space and let $a, b, c\in X$.
Let $r = \max \{ d(a,c), d(b,c) \}$.
We define $\mathrm{div} (a, b, c)$ to be the infimum of the lengths of paths joining $a$ and $b$ which lie outside of $\mathrm{int} B(c,r)$.
\end{defn}

\begin{theorem}
\label{spread}
Suppose that $(M,g)$ is a radially symmetric plane with its pole at $p \in M$ where $g = dr^2 + f(r)^2 d\theta^2$ where $\lim_{r \to \infty} \frac{f(r)}{r} = \infty$.
Then $\mathrm{Con}_\omega(M, \mathrm{p}, (r_n))$ is not proper and $\mathrm{p}$ is a cut point for every non-principal ultrafilter $\omega$ and every sequence of scaling factors $(r_n)$ with $\lim_\omega r_n = \infty$.
\end{theorem}

\begin{proof}
Lemma \ref{focal} allows us to assume without losing generality that $f$ is strictly monotone.
Choose an arbitrary non-principal ultrafilter $\omega$ and choose a sequence of positive numbers $(r_n)$ for which $\lim_\omega r_n = \infty$.
$\lim_{r \to \infty} \frac{f(r)}{r} = \infty$ implies that $\lim_\omega \frac{f(r_n)}{r_n} = \infty$.

Let $\mathrm{a} =  (a_n) \in \mathrm{Con}_\omega(M, \mathrm{p}, (r_n))$ be such that $ d_\omega(\mathrm{a},\mathrm{p}) = D > 0$.
As $d_\omega ( \mathrm{a}, \mathrm{p}) > 0$, $\omega-$all $a_n$ can be written as $(Dr_n, \theta_n)$ where $\theta_n \in [0, 2 \pi]/(0 \sim 2\pi) = \mathbb{S}^1$. 

Recall that $\mathbb{S}^1$ acts on $M \setminus p$ by $\theta \circ (r, \phi) = (r, \theta + \phi \mod{2 \pi})$ and we can continuously extend this to $M$ by having $\theta \circ p = p$ for all $\theta \in \mathbb{S}^1$.

It follows that $\mathbb{S}^1$ acts on $\mathrm{Con}_\omega(M, \mathrm{p}, (r_n))$ by $\theta \circ \mathrm{a} = ( \theta \circ a_n)$.
As $\lim_{\omega} d(p, a_n) = \infty$, it follows that $\lim_\omega \frac{\gamma_n}{r_n} = \infty$ where $a_n = (Dr_n, \theta_n) $ and $\gamma_n$ is the shortest path in $M \setminus B(p,r_n)$ which joins $a_n$ and $\theta \circ a_n$.
This implies that $\lim_\omega \frac{\mathrm{div}(a_n, \theta \circ a_n, p)}{r_n} = \infty$.
Therefore from lemma 3.14 of \cite{DrutuMozesSapir}, we see that $\mathrm{p}$ is a cut-point of $\mathrm{Con}_\omega(M, \mathrm{p}, (r_n))$.

Let $\gamma$ be the geodesic segment $\gamma : [0, d_\omega ( \mathrm{a}, \mathrm{p})] \to M$ where $\gamma(t) = (\exp_p(t \exp_p^{-1}a_n)) \in \mathrm{Con}_\omega (M,(r_n))$ and let 
\begin{equation*}
    \mathcal{H} = \bigcup_{\theta \in UT_pM} \theta \circ \gamma
\end{equation*}
We note that $\mathcal{H}$ is a metric embedding of the $\mathfrak{c}$-hedgehog space in $\mathrm{Con}_\omega(M, (r_n))$ which ensures that it is not proper as $\mathcal{H} \subset \mathrm{Con}_\omega (M, (r_n))$ is a closed and bounded subset which is not compact.
\end{proof}

We saw in the previous theorem that $(M,g)$ not being monodirectional forces the base point $\mathrm{p}$ to be a cut point in all asymptotic cones $\mathrm{Con}_\omega (M, \mathrm{p}, (r_n)) $ with $\mathrm{p}$ as the base point.
Naturally, we shall see in the next theorem that $(M,g)$ not being monodirectional is a necessary and sufficient condition for $\mathrm{p}$ to be a cut point in all asymptotic cones $\mathrm{Con}_\omega (M, \mathrm{p}, (r_n))$.

\begin{theorem}
Suppose that $(M,g)$ is a radially symmetric plane with its pole at $p \in M$ where $g = dr^2 + f(r)^2 d\theta^2$ where $\liminf_{r \to \infty} \frac{f(r)}{r} < \infty$. 
Then there exists a sequence of scaling factors $(r_n)$ with $\lim_{n \to \infty} r_n = \infty$ such that $\mathrm{p} \in \mathrm{Con}_\omega (M, (r_n))$ is not a cut point for all non-principal ultrafilters $\omega$.
\end{theorem}

\begin{proof}
Let $\alpha, \beta : [0, \infty) \to (M,g)$ be two geodesic rays with $\alpha(0) = \beta(0) = p $.
As $g = dr^2 + f(r)^2 d\theta^2$ with $\liminf_{r \to \infty} \frac{f(r)}{r} = \lambda $ where $\lambda \in [0, \infty)$, there exists for all $\epsilon > 0$ and for all $R > 0$, an $r \geq R$ such that $\frac{f(r_n)}{r_n} \leq \lambda + \epsilon$.

We now fix an $\epsilon > 0$ and for each $n \in \mathbb{N}$, choose an $r_n \geq n$ such that $\frac{f(r_n)}{r_n} \leq \lambda + \epsilon$.
We may assume without losing generality that $r_{n +1} > r_n$ for all $n \in \mathbb{N}$ and $\lim_{n \to \infty}r_n = \infty$.
As $\lim_{n \to \infty}r_n = \infty$, $\lim_\omega r_n = \infty$ for all non-principal ultrafilters.

We consider $\mathrm{Con}_\omega (M, \mathrm{p}, (r_n))$ and denote as $\mathrm{a} = (a_n)$ where $a_n = \alpha(r_n)$ for all $n \in \mathbb{N}$.
Similarly define $\mathrm{b}= (b_n)$ where $b_n = \beta(r_n)$ for all $n \in \mathbb{N}$.
We see that $\mathrm{div}(a_n, b_n, p,1) \leq f(r_n) \leq (\lambda + \epsilon)r_n$ for all $n \in \mathbb{N}$.
It follows that
\begin{equation*}
    \lim_\omega \frac{\mathrm{div}(a_n, b_n, p, 1)}{r_n} \leq \lambda + \epsilon < \infty
\end{equation*}
which taken together with lemma 3.14 of \cite{DrutuMozesSapir} implies that $\mathrm{p}$ is not a cut-point.
\end{proof}

We shall now study when an asymptotic cone $\mathrm{Con}_\omega (M, \mathrm{p}, (r_n))$ is a proper geodesic metric space and when it isn't and use these and the lack of cut points to uniquely characterise the Euclidean plane upto quasi-isometry in terms of its asymptotic cones.

\begin{lemma}
\label{squeeze}
Let $(M,g)$ be a radially symmetric plane with $g = dr^2 + f(r)^2 d\theta^2$.
If $\liminf_{r \to \infty} \frac{f(r)}{r} = 0$, then there exists a non-principal ultrafilter $\omega$ and a sequence of scaling factors $(r_n)$ with $\lim_\omega r_n = \infty$ such that $\mathrm{Con}_\omega (M, (r_n))$ has cut points.
\end{lemma}

\begin{proof}
Let $\omega$ be an arbitrary non-principal ultrafilter and 
Suppose that $\liminf_{r \to \infty} \frac{f(r)}{r} = 0 $.
This implies that there exists a $D > 0$ such that $\lim_{\omega} \frac{f(Dr_n)}{r_n} = 0$.
Now choose $\mathrm{a}, \mathrm{b} \in \mathrm{Con}_\omega (M, (r_n))$ with $d_\omega (\mathrm{a}, \mathrm{p}) = d_\omega( \mathrm{b}, \mathrm{p}) = D$.
We may assume without losing generality that $\mathrm{a} = (Dr_n, \theta_n)$ and $\mathrm{b}= ( Dr_n, \phi_n)$.
We see that 
\begin{equation*}
d_\omega( \mathrm{a}, \mathrm{b}) = \lim_\omega \frac{d(( Dr_n \theta_n),( Dr_n, \phi_n))}{r_n} \leq \lim_\omega \frac{f(Dr_n)}{r_n} = 0
\end{equation*}
This implies that $\partial B( \mathrm{p}, D)$ is a single point and a cut point in $\mathrm{Con}_\omega (M, (r_n))$.
\end{proof}

\begin{theorem}
\label{grom}
Let $(M,g)$ be a radially symmetric Riemannian plane with bounded geometry with $g = dr^2 + f(r)^2 d\theta^2$.
Then every asymptotic cone of $(M,g)$ is a proper geodesic space if and only if $\limsup_{r \to \infty} \frac{f(r)}{r} < \infty$.
\end{theorem}

\begin{proof}
Suppose that $\limsup_{r \to \infty} \frac{f(r)}{r} = \infty$.
There exists a non-principal ultrafilter $\omega$, a sequence of scaling factors $(r_n)$ and $D > 0$ such that
\begin{equation*}
    \lim_\omega \frac{f(Dr_n)}{r_n} = \infty
\end{equation*}
Denote as $\partial B( \mathrm{p}, D) = \{ \mathrm{a} \in \mathrm{Con}_\omega (M, \mathrm{p}, (r_n)) : d_\omega ( \mathrm{p}, \mathrm{a}) = D \}$.
We see that $\partial B(\mathrm{p}, D)$ is a closed subset of $\mathrm{Con}_\omega (M, \mathrm{p}, (r_n))$.

Let $\mathrm{a} = (a_n) = (Dr_n, \theta_n) \in \partial B(\mathrm{p}, D)$.
Let $\theta \circ \mathrm{a} = ( Dr_n, \theta \circ \theta_n)$ where $\theta \in \mathbb{S}^1$.
As $\lim_\omega \frac{f(Dr_n)}{r_n} = \infty$, we see that $\{ \theta \circ \mathrm{a} \}_{\theta \in \mathbb{S}^1} \subset \partial B ( \mathrm{p}, D)$ is an infinite discrete set.
This ensures the non-compactness of $\partial B ( \mathrm{p}, D)$ which ensures that $\mathrm{Con}_\omega ( M, \mathrm{p}, (r_n))$ is not proper.

Conversely suppose that $\limsup_{r \to \infty} \frac{f(r)}{r} < \infty$.
This implies that 
\begin{equation*}
    \limsup_{r\to \infty} \frac{f(Dr)}{r} =D  \limsup_{r \to \infty} \frac{f(Dr)}{Dr} < \infty
\end{equation*}
This implies that for every non-principal ultrafilter $\omega$ and every sequence $(r_n)$ with $\lim_\omega r_n = \infty$.

For each $D >0$, we may denote $S_D \subset  \mathfrak{X} $ where 
\begin{equation*}
A \in S_D \text{ iff } A = \frac{1}{r_n} B(p,r_n) \text{ upto isometry}
\end{equation*}
$\limsup_{r \to \infty} \frac{f(Dr_n)}{r} < \infty$
\end{proof}
Since $S_D$ are all uniformly totally bounded for all $D >0$ ( see definition 7.4.13 of \cite{BuragoIvanov}), this makes all $B(\mathrm{p}, D)$ compact ( from  theorem 7.4.15 of \cite{BuragoIvanov}).

Therefore every asymptotic cone of a radially symmetric $(M,g)$ with $g = dr^2 + f(r)^2 d\theta^2 $ is a proper metric space if and only if $\lim_{r \to \infty} \frac{f(r)}{r} = \infty$.

\textbf{Convention}- If we mention $ \mathrm{Con}_\omega (\mathbb{R}^2, \mathrm{x}, (r_n))$ without the metric being mentioned explicitly, we shall take $\mathbb{R}^2$ to be the Euclidean plane.

\begin{theorem}
\label{rtwo}
Let $(M,g)$ be a Riemannian plane with bounded geometry.
Every asymptotic cone of $(M,g)$ is a proper geodesic metric space without cut points if and only if $(M,g)$ is quasi-isometric to the Euclidean plane.
\end{theorem}

\begin{proof}
If $(M,g)$ is quasi-isometric to the Euclidean plane, then proposition 2.5 of \cite{Riley} ensures that every asymptotic cone $\mathrm{Con}_\omega (M, \mathrm{x}, (r_n))$ is homeomorphic to $\mathrm{Con}_\omega ( \mathbb{R}^2, \mathrm{y}, (d_n))$ where there is a $(q,Q)-$quasi-isometry $\varphi : M \to \mathbb{R}$ with $\varphi ( \mathrm{x}) = \mathrm{y}$.

$( \mathbb{R}^2, g_{\mathbb{R}^2})$ being a Lie group allows to assume without losing generality that $\mathrm{y} = \mathrm{p} = (p_n)$ where $p_n = 0 \in \mathbb{R}^2$ for all $n \in \mathbb{N}$.
Construct a map $f : \mathbb{R}^2 \to \mathrm{Con}_\omega (\mathbb{R}^2, \mathrm{p}, (d_n))$ where
\begin{equation*}
    f( r, \theta) = (rd_n , \theta)
\end{equation*}
Lemma \ref{coord} allows us to assume that any $\mathrm{q} \in \partial B(\mathrm{p}, r) \subset \mathrm{Con}_\omega (\mathbb{R}^2, \mathrm{p}, (d_n))$ is of the form $\mathrm{q}= (rd_n, \theta_n)$.
We also see that if $\lim_\omega \theta_n = \theta$, then 
\begin{equation*}
    d_\omega ((rd_n,\theta_n), (rd_n \theta)) \leq  \lim_\omega |\theta_n - \theta| \frac{f(rd_n)}{d_n} = r \lim_\omega |\theta_n -\theta| \frac{f(rd_n)}{rd_n} = r\lim_\omega | \theta_n - \theta| = 0
\end{equation*}
We may therefore parametrise each $\mathrm{q} \in \partial B ( \mathrm{p}, r)$ as $\mathrm{q} = (rd_n, \theta)$ for some $\theta \in \mathbb{S}^1$.

Conversely, define $f' : \mathrm{Con}_\omega(\mathbb{R}^2, \mathrm{p}, (d_n)) \to \mathbb{R}^2$ by 
\begin{equation*}
    f'(r d_n, \theta) = (r, \theta) 
\end{equation*}
We see that $f' \circ f : \mathbb{R}^2 \to \mathbb{R}^2$ is the identity mapping.
Therefore, $\mathrm{Con}_\omega (\mathbb{R}^2, \mathrm{p}, (d_n))$ is isometric to $\mathbb{R}^2$.

To show the converse, we shall assume ( thanks to theorem \ref{radsim} and proposition  2.5 of \cite{Riley}), that $(M,g)$ is a radially symmetric plane with $g=dr^2 + f(r)^2d\theta^2$.
Theorem \ref{grom} rules out the case where $\limsup_{r \to \infty} \frac{f(r)}{r}= \infty$ and lemma \ref{squeeze} rules out the case where $\liminf_{r \to \infty} \frac{f(r)}{r} = \infty$.

Therefore, there are $0 < \lambda_1 \leq \lambda_2 < \infty$ such that
\begin{equation*}
    \lambda_1 r \leq f(r) \leq \lambda_2 r
\end{equation*}
for all $r \geq 0$.
We define a map $\varphi : (M,g) \to \mathbb{R}^2$ where $\varphi ((r,\theta)_M) = (r, \theta)_{\mathbb{R}^2}$.
Let $q,s \in (M,g)$ and $\gamma$ be a path in $(M,g)$ with $q$ and $s$ as their end points.
We see that $\lambda_1 r \leq f(r) \leq \lambda_2 r$ for all $r>0$ implies that 
\begin{equation*}
    \lambda_1 l( \varphi(\gamma)) \leq l(\gamma) \leq \lambda_2 l(\varphi(\gamma)) 
\end{equation*}
This implies that $\varphi : M \to \mathbb{R}^2$ is a $\max \{ \lambda_1, \lambda_2, \frac{1}{\lambda_1}, \frac{1}{\lambda_2 \}}-$quasi-isometry which implies the converse.
\end{proof}

\section{Acknowlegdement}

The author thanks Siddhartha Gadgil for introducing him to quasi-redirection, for going through earlier drafts of this paper and for his helpful feedback.
The author also thanks Alexandre Eremenko and Sam Nead for pointing out examples of planes which are conformally equivalent to $\mathbb{H}^2$ which is not Gromov hyperbolic.
This research was supported in part by the International Centre for Theoretical Sciences (ICTS) for the programme "Geometry and Analysis of Minimal Surfaces" , (code: ICTS/GAMS2025/08) and for the programme "New trends in Teichmüller theory (code: ICTS/nteich2025/02)" .

\end{document}